\documentclass[11pt,a4paper,reqno]{amsart}
\usepackage{amsmath,amssymb, amsbsy}
\usepackage{color,psfrag}
\usepackage[dvips]{graphicx}
\usepackage{color}

\theoremstyle{plain}
\newtheorem{thm}{Theorem}[section]
\theoremstyle{plain}
\newtheorem{lem}[thm]{Lemma}

\newtheorem{cor}[thm]{Corollary}

\theoremstyle{definition}

\newtheorem{rem}{Remark}[section]

\newcommand{\loc}{{\mathrm{loc}}}

\newcommand{\rn}{\mathbb{R}^{N}}

\newcommand{\hn}{\mathbb{H}^{N}}
\newcommand{\dvh}{\mbox{d}v_{\mathbb{H}^{N}}}
\newcommand{\hnn}{\mathbb{H}^{2}}

\newcommand{\authorfootnotes}{\renewcommand\thefootnote{\@fnsymbol\c@footnote}}%

\newcommand{\Hmm}[1]{\leavevmode{\marginpar{\tiny%
$\hbox to 0mm{\hspace*{-0.5mm}$\leftarrow$\hss}%
\vcenter{\vrule depth 0.1mm height 0.1mm width \the\marginparwidth}%
\hbox to
0mm{\hss$\rightarrow$\hspace*{-0.5mm}}$\\\relax\raggedright #1}}}

\numberwithin{equation}{section} \allowdisplaybreaks

\usepackage{fancyhdr}

\usepackage{calc}
\usepackage{url}

\usepackage[text={6in,9in},centering]{geometry}

\begin{document}
        \title[An optimal improvement for the Hardy inequality]{An optimal improvement for the Hardy inequality \\ on the hyperbolic space and related manifolds}

\date{}

\author[Elvise BERCHIO]{Elvise BERCHIO}
\address{\hbox{\parbox{5.7in}{\medskip\noindent{Dipartimento di Scienze Matematiche, \\
Politecnico di Torino,\\
        Corso Duca degli Abruzzi 24, 10129 Torino, Italy. \\[3pt]
        \em{E-mail address: }{\tt elvise.berchio@polito.it}}}}}

        \author[Debdip GANGULY]{Debdip GANGULY}
\address{\hbox{\parbox{5.7in}{\medskip\noindent{Department of Mathematics,\\
 Indian Institute of Science Education and Research,\\
 Dr. Homi Bhabha Road, Pashan\\
        Pune 411008, India. \\[3pt]
         \em{E-mail address: }{\tt debdipmath@gmail.com}}}}}

\author[Gabriele GRILLO]{Gabriele GRILLO}
\address{\hbox{\parbox{5.7in}{\medskip\noindent{Dipartimento di Matematica,\\
Politecnico di Milano,\\
   Piazza Leonardo da Vinci 32, 20133 Milano, Italy. \\[3pt]
        \em{E-mail addresses: }{\tt
          gabriele.grillo@polimi.it}}}}}

      \author[Yehuda PINCHOVER]{Yehuda PINCHOVER}
\address{\hbox{\parbox{5.7in}{\medskip\noindent{Department of Mathematics, \\
Technion - Israel Institute of Technology,\\
   Haifa 3200003, Israel. \\[3pt]
        \em{E-mail address: }{\tt pincho@technion.ac.il}}}}}

\date{\today}



\begin{abstract}

 We prove \emph{optimal} improvements of the Hardy inequality on the hyperbolic space.
 Here, optimal means that
 the resulting operator is \emph{critical} in the sense of \cite{pinch}, namely the associated inequality cannot be further improved.
 Such inequalities arise from more general, \emph{optimal} ones valid for the operator $ P_{\lambda}:= -\Delta_{\hn} - \lambda$ where $0 \leq  \lambda \leq \lambda_{1}(\hn)$ and $\lambda_{1}(\hn)$ is the bottom of the $L^2$ spectrum of $-\Delta_{\hn} $, a problem that had been studied in \cite{BGG} only for the operator $P_{\lambda_{1}(\hn)}$. A different, critical and new inequality on $\hn$, locally of Hardy type, is also shown. Such results have in fact greater generality since they are proved on general Cartan-Hadamard manifolds under curvature assumptions, possibly depending on the point. Existence/nonexistence of extremals for the related Hardy-Poincar\'e inequalities are also proved using
 concentration-compactness technique and a Liouville comparison theorem. As applications of our inequalities we obtain an improved Rellich inequality and we derive a quantitative version of Heisenberg-Pauli-Weyl uncertainty principle for the operator $P_\lambda.$
\end{abstract}

\maketitle

 \section{Introduction}

 The Hardy inequality on (Euclidean) domains has been studied intensively for the last few decades. Much of the interest has centered on \emph{optimal} improvements of the inequality and the effect of the domain on the Hardy constant. Its generalization to Riemannian manifolds was intensively pursued  after the seminal work
 of Carron \cite{Carron}, see for instance \cite{BGG, BAGG, Dambrosio, Kombe1, Kombe2, AKR, Yang}. Let $(M, g)$ be a Riemannian manifold and let $\varrho(x)$ be a weight function satisfying the Eikonal equation $|\nabla_g \varrho| = 1$ and
 $\Delta_{g} \varrho \geq \frac{C}{\varrho}$ where $C > 0$ a positive constant. By \cite{Carron} there holds

 \begin{equation}\label{General_Hardy_inequality}
 \int_{M} |\nabla_{g} u|^2 \, {\rm d}v_g \geq \left(\frac{C-1}{2} \right)^2   \int_{M} \frac{u^2}{\varrho^2} \, {\rm d}v_g \quad \forall  \ u \in C_{c}^{\infty}(M \setminus \varrho^{-1}\{ 0 \}).
 \end{equation}

 In the case of a Cartan-Hadamard manifold $M$ of dimension $N$ (namely, a manifold which is complete, simply-connected, and has everywhere non-positive sectional curvature),
 the geodesic distance function ${\rm d}(x, x_0),$ where $x_0 \in M$, satisfies all the assumptions of the weight $\varrho$ and the above inequality holds with \emph{best} constant $\left(\frac{N-2}{2} \right)^2$, see \cite{AKR}. In particular, considering the most important example of Cartan-Hadamard manifold, namely the hyperbolic space $\hn$, inequality \eqref{General_Hardy_inequality} reads


 \begin{equation}\label{hardy}
 \int_{\hn} |\nabla_{\hn} u|^2 \ \dvh \geq \left(\frac{N-2}{2} \right)^2   \int_{\hn} \frac{u^2}{r^2} \ \dvh, \quad \forall \ u \in C_{c}^{\infty}(\hn \setminus \{ x_0 \})
 \end{equation}
 with $r:={\rm d}(x, x_0)$ and $x_0 \in \hn$ is a fixed pole.

  \medskip


The effect of curvature has been exploited in \cite{Kombe1, Kombe2,AKR, Yang} to improve inequality \eqref{General_Hardy_inequality} (in the sense of adding nonnegative terms in the right side of the inequality) on Cartan-Hadamard manifolds. This is in contrast to what happens in the Euclidean setting where the operator $-\Delta_{\mathbb{R}^{N}} - \left(\frac{N-2}{2} \right)^2  \frac{1}{|x|^2} $ is known to be  \emph{critical} in $\mathbb{R}^{N} \setminus \{ 0 \}$ (see \cite{pinch}) and improvements of such quadratic form inequality are not possible. However, there is a huge literature about improved Hardy inequalities on \emph{bounded} Euclidean domains after the seminal works of Brezis and Marcus \cite{BrezisM} and Brezis and Vazquez \cite{Brezis}. See also \cite{GFT, BFT2, BT, FMT, FT, FTT, gaz, GM} and references therein. We now describe qualitatively the contributions given in the present paper.

\noindent $\bullet$ \bf Critical improvements of the Hardy inequality with optimal constant\rm.  It is known that the operator, $-\Delta_{\hn} - \left(\frac{N-2}{2} \right)^2  \frac{1}{r^2} $ is \emph{subcritical} operator in $\hn \setminus \{ x_0 \},$ and the existence of a remainder term for inequality \eqref{hardy} involving a multiple of the $L^2$-norm is also known by \cite{Yang}. Furthermore, a new type of improvement of \eqref{hardy}, and more generally of \eqref{General_Hardy_inequality} on Cartan-Hadamard manifolds, has been recently provided in \cite{AKR} by showing that more curvature implies more powerful improvements, see Remark \ref{yang} below. Nevertheless, as far we are aware, the criticality of the resulting \lq\lq improved" operators has never been studied.

\medskip

The first goal of the present paper is to address this topic by looking for a weight $V \geq 0$ such that the following improved Hardy inequality holds true
  \begin{equation}\label{initial}
 \int_{\hn} |\nabla_{\hn} u|^2 \ \dvh \geq \left(\frac{N-2}{2} \right)^2  \int_{\hn} \frac{u^2}{r^2} \ \dvh  + \int_{\hn} Vu^2 \  \dvh \quad \forall u \in C_{c}^{\infty}(\hn)\,
   \end{equation}
and the associated operator $-\Delta_{\hn} - \left(\frac{N-2}{2} \right)^2  \frac{1}{r^2} -V$ is\emph{ critical} in $\hn \setminus \{ x_0 \}$. Hence, the inequality is not true when $V$ is replaced by $W \ge V$, $W\not=V$, and this is the reason why we will call such $V$ an \emph{optimal}  weight. In this respect we note that for any second-order elliptic subcritical operator $P$ in $\hn$, and any compactly supported, positive perturbation  $V$ of $P$ in $\hn$, there always exists $\lambda_0$ s.t. $P-\lambda_0V$ is critical in $\hn$ (see \cite{pinch-duke}). So qualitatively we aim at finding a potential that is as large as possible at infinity and such that inequality \eqref{initial} is not improvable.
\medskip\par
In Corollary~\ref{improved-hardy_1} below we show that an optimal radial weight $V\geq 0$ such that \eqref{initial} holds is
    \begin{equation}\label{hardy_weight}
  V(r) = (N-2) + \frac{(N-2)(N-3)}{4} \, g(r)\,,
     \end{equation}
where $g(r)= \frac{r \coth r - 1}{r^2} > 0$ and $r>0$. In particular, $g$ satisfies $g(r) \sim \frac{1}{3} $ as $r \rightarrow 0^+$ and $g(r) \sim  \frac{1}{r}$ as $r \rightarrow +\infty.$ It is clear from \eqref{hardy_weight} that $V(r)$ yields, as a byproduct, an $L^2$ improvement of the Hardy inequality \eqref{hardy} and we point out that,  to our knowledge, the constant  $N-2$ we get in front of the $L^2$-term is greater than the existing known bounds in literature, cf. \cite{Yang}. Though, except for $N= 3$, the optimality of the weight $V$ does not imply that $N-2$ is the \emph{best} constant in obvious sense. It is also interesting to note that our optimal inequality is closely related
to the improved Hardy inequality studied in \cite{AKR}, we refer to Remark \ref{yang} for a detailed discussion. Here we only mention that the main result on the Hardy inequality given in \cite{AKR}, when considered on $\hn$, follows as a particular case of our results. Also the extension of our results to more general Cartan-Hadamard manifolds is obtained under less restrictive assumptions than in \cite{AKR}. Indeed, we only require curvature bound in the \emph{radial direction}, see Section \ref{manifolds} and, besides, we allow for curvature bounds varying with the point.


  \medskip

\noindent $\bullet$ \bf Hardy-type improvements of the Poincar\'e inequality\rm. It is worth noting that the weight $V(r)$ in \eqref{hardy_weight} originates from a suitable family of Hardy weights improving Poincar\'e-type inequalities on $\hn$ with $N\geq 2$. Indeed, the validity of the Poincar\'e inequality (or $L^2$-gap inequality, see \cite{NN} for generalizations) on $\hn$ with best constant

 \begin{equation}\label{poincare}
  \lambda_{1}(\hn): = \inf_{u \in C_{c}^{\infty}(\hn)  \setminus \{ 0 \}} \dfrac{\int_{\hn} |\nabla_{\hn} u|^2 \ \dvh}{\int_{\hn} u^2 \ \dvh}=\left(\frac{N-1}{2} \right)^2 \,,
  \end{equation}
makes it natural to inquire whether, for any given $\lambda \leq   \lambda_{1}(\hn)$, a Hardy-type inequality associated to the  family of nonnegative operators $P_{\lambda}:=-\Delta_{\hn} -\lambda$ holds. More precisely, for any $\lambda \leq   \lambda_{1}(\hn)$, one looks for functions $V_{\lambda} \geq 0$ such that the following inequality holds true
  \begin{equation}\label{gap_lambda}
 \int_{\hn}|\nabla_{\hn} u|^2\, {\rm d}v_{\hn} -\lambda \int_{\hn}  u^2\, {\rm d}v_{\hn}\geq \int_{\hn} V_{\lambda}\,u^2 \  \dvh\ \ \ \forall u\in C_c^\infty(\hn),
 \end{equation}
 and the operator $P_{\lambda}-V_{\lambda}$ is critical in $\hn \setminus \{ x_0 \}$ so that \eqref{gap_lambda} does not hold for any $W_{\lambda} \ge V_{\lambda}$, $W_{\lambda} \not= V_{\lambda}$.\par

 \medskip

  When $\lambda=  \lambda_{1}(\hn)$ and $N\geq 3$, a weight such that the above condition is satisfied is known to exist. More precisely,  inequality \eqref{gap_lambda} holds with $\lambda= \lambda_{1}(\hn)$ and
   \begin{equation}\label{V_lambda1}
 V_{ \lambda_{1}(\hn)}(r)=\frac{1}{4}  \frac{1}{r^2} + \frac{(N-1)(N-3)}{4}  \frac{1}{\sinh^2 r}\,.
  \end{equation}
 Furthermore, the operator $P_{ \lambda_{1}(\hn)}- V_{ \lambda_{1}(\hn)}$ is critical in $\hn$. The inequality has been shown first in \cite{AK} and then, with different methods, adaptable to larger classes of manifolds, in \cite{BGG}, where criticality has also been shown. We refer the interested reader to \cite{BG1} and \cite{BAGG} for higher order and $L^p$ version of inequality \eqref{gap_lambda} for $\lambda= \lambda_1,$ respectively, and to \cite{BRM2} for other functional inequalities in the same setting but involving the Green's function of the Laplacian.
 \par\medskip

Hence, a further goal of this work is to complete the study of \eqref{gap_lambda} for $\lambda<  \lambda_{1}(\hn)$ and to address the criticality issue when $N=2$, a case which was not dealt with in \cite{BGG}. Clearly, from the validity of \eqref{gap_lambda} with $\lambda= \lambda_{1}(\hn)$ and  $V_{\lambda}=V_{ \lambda_{1}(\hn)}$ as given above, it is readily deduced that for any $ \lambda <  \lambda_{1}(\hn)$ an optimal radial weight for $P_{\lambda}$ is $\overline V_{\lambda}(r)=(\lambda_{1}(\hn)-\lambda)+V_{ \lambda_{1}(\hn)}(r)$. In Theorem \ref{improved-hardy} below we provide a second optimal radial weight $V_{\lambda}$ which coincides with $V_{ \lambda_{1}(\hn)}$ if $\lambda=  \lambda_{1}(\hn)$, while it gives inequality \eqref{initial} with the weight in \eqref{hardy_weight} if $\lambda=N-2$. Moreover, for $N\geq 3$ and any $ \lambda \leq  \lambda_{1}(\hn)$, $V_{\lambda}$ satisfies
   $$
  V_{\lambda}(r) \sim \left(\frac{N-2}{2} \right)^2  \frac{1}{r^2} \quad \text{as } r \rightarrow 0^{+}\,.
  $$
 The same asymptotic holds for $\overline V_{\lambda}$, hence both $V_{\lambda}$ and $\overline V_{\lambda}$ tend to reproduce the classical Hardy weight near the origin but it can be shown that $V_{\lambda}$ is larger than $\overline V_{\lambda}$, see Remark \ref{asym} below. Clearly, when $N=2$ one cannot expect an improvement with a Hardy term like in higher dimensions. Indeed, near the origin we have
 $$V_{ \lambda}(r)\sim \frac{(1+\sqrt{1-4\lambda})(1+3\sqrt{1-4\lambda})}{12}   \quad \text{as } r \rightarrow 0^{+}\,$$
 for any $\lambda\leq  \lambda_{1}(\hnn)=\frac{1}{4}$.

 \medskip

 \noindent $\bullet$ \bf A new critical quadratic form inequality on the hyperbolic space\rm. We shall show the validity of a new quadratic form inequality on $\hn$, which is \it locally \rm of Hardy type. The inequality reads
\begin{equation}\label{new}
 \begin{aligned}
\int_{\hn} |\nabla_{\hn} u|^2 \, {\rm d}v_{\hn} & \geq \left(\frac{N-2}{2} \right)^2 \int_{\hn} \frac{u^2}{\sinh^2 r} \, {\rm d}v_{\hn} +  \frac{1}{4}
\int_{\hn} \frac{u^2}{\sinh^2 r (\log (\tanh \frac{r}{2}))^2} \, {\rm d}v_{\hn}  \\
& + \frac{N(N-2)}{4} \int_{\hn}  \, u^2 \, {\rm d}v_{\hn}\,.
\end{aligned}
\end{equation}

It will also be shown that the operator

\begin{equation*}
-\Delta_{\hn} - \left(\frac{N-2}{2} \right)^2 \,\frac{1}{\sinh^2 r} \, - \, \frac{1}{4\sinh^2 r (\log (\tanh \frac{r}{2}))^2}   - \frac{N(N-2)}{4}
\end{equation*}
is critical in $\hn \setminus \{ x_0 \}$ and  the constant $\frac{N(N-2)}{4}$ is
sharp in the obvious sense. For a somewhat related inequality on the geodesic ball and for \it radial \rm functions,  see \cite[Prop. 1.8]{CS}, optimality issues not being discussed there.

 \medskip

 \noindent $\bullet$ \bf General Cartan-Hadamard manifolds\rm. It is important to comment that \it all \rm the above results in fact hold under the curvature \it bound \rm $K_R\le-1$, $K_R$ being the sectional curvature in the radial direction of a Cartan-Hadamard manifold with a pole (or, with some modifications, if $K_R\le-c<0$), see Theorem \ref{general manifold2} and its Corollaries. We have so far stated them in the special case of ${\mathbb H}^n$ for greater readability only. In fact, Theorem \ref{general manifold2} proves suitable integral inequalities even under more general curvature bounds that can depend on the point. Inequality \eqref{new} can be extended to general Cartan-Hadamard manifolds as well, in fact a new critical inequality is proved in Theorem \ref{improved-hardy_general}. It is important to stress that such inequality will be shown under the assumption that curvature is \it strictly negative \rm at infinity, more precisely it can be allowed to vanish as the distance from a given pole tends to infinity but not faster than quadratically.

  \medskip

\noindent $\bullet$ \bf Existence of extremals for optimal inequalities\rm. Coming back to inequality \eqref{gap_lambda} with $N\geq 3$, we also take the different attitude of fixing $V_{\lambda}(r)=\frac{I(\lambda)}{r^2} $ and looking for the best constant $I=I(\lambda)>0$ such that \eqref{gap_lambda} holds. In other words, the following infimum problem arises

\begin{equation}\label{quotient}
I(\lambda) : = \inf_{u \in C_{c}^{\infty} (\hn) \setminus \{ 0 \}} \dfrac{\int_{\hn} |\nabla u|^2 \ \dvh - \lambda \int_{\hn} u^2 \ \dvh}{\int_{\hn} \frac{u^2}{r^2} \ \dvh}\,.
\end{equation}

Clearly, $I(0)=\left(\frac{N-2}{2} \right)^2$ while, by \eqref{V_lambda1}, $I(\lambda_1(\hn))=\frac{1}{4}.$
  \medskip

In Theorems \ref{maintheorem1} and \ref{maintheorem2} we investigate existence/non existence of extremals of $I(\lambda)$ for any $\lambda\in [0,\lambda_1(\hn)]$. Furthermore, we provide a lower and an upper bound of the maximum value of $\lambda$ such that $I(\lambda)=\left(\frac{N-2}{2} \right)^2$, namely of the best constant in front of the $L^2$-type remainder term for \eqref{hardy}.

  \medskip
\noindent $\bullet$ \bf Further results\rm. The rest of the paper is, on one hand, devoted to present a further remarkable application of \eqref{gap_lambda}, namely the derivation of suitable quantitative versions of \emph{Heisenberg-Pauli-Weyl uncertainty principle} for the \it shifted \rm Laplacian in the hyperbolic setting; the corresponding inequalities should be compared with those obtained in \cite{Kombe1, Kombe2, AKR}. Besides, we also generalize the Hardy-type inequalities to more general ones in which the energy term may involve weights as well, and also prove improved, weighted Rellich inequalities in the spirit of \cite{BGG}, with optimal Rellich term.

\medskip

\noindent $\bullet$ \bf Plan of the paper\rm. The paper is organized as follows. In Section~\ref{main} we state Theorem~\ref{improved-hardy} in $\hn$, namely our family of optimal inequalities \eqref{gap_lambda}, and some interesting inequalities derived  from Theorem~\ref{improved-hardy}, among which the inequality \eqref{initial} associated to the weight \eqref{hardy_weight}. Finally we state Theorems \ref{maintheorem1} and \ref{maintheorem2} related to the study of existence/non existence of extremals for \eqref{quotient}. Section \ref{HPW} is devoted to the  application of Theorem~\ref{improved-hardy} to obtain the above mentioned quantitative versions of Heisenberg-Pauli-Weyl uncertainty principle involving the shifted Laplacian in the hyperbolic space setting.  In Section~\ref{manifolds}, we discuss the extension of our results to general Cartan-Hadamard manifolds. Sections~\ref{proof_thm}, ~\ref{proof_thm_1} and ~\ref{proof_general_manifolds} are devoted to the proofs of the statements of Sections~\ref{main} and \ref{manifolds}. Finally, in the Appendix we state some Hardy-Maz'ya type inequalities in dimension 2 related with the inequalities of Section~\ref{main}.


\section{Main results}\label{main}

We start by providing a suitable family of optimal Hardy weights for the operators $P_{\lambda}:=-\Delta_{\hn} -\lambda$. We comment here and once for all that, although stated for functions compactly supported away from the pole, most inequalities also holds without such requirement by density arguments: in fact, e.g. in the next Theorem formula \eqref{improved-poinc-lambda} holds without such requirement if $N\ge3$.

\begin{thm}\label{improved-hardy}
Let $N \geq 2.$ For all $ \lambda \leq  \lambda_{1}(\hn)=\left(\frac{N-1}{2} \right)^2$ and all $u \in C_{c}^{\infty} (\hn \setminus \{ x_0 \})$ there holds
\begin{equation}\label{improved-poinc-lambda}
\begin{aligned}
 &\int_{\hn}|\nabla_{\hn} u|^2\, \emph{d}v_{\hn} -\lambda \int_{\hn}  u^2\, \emph{d}v_{\hn}\\
 & \geq \frac{(\gamma_{N}(\lambda)+1)^2}{4} \int_{\hn} \frac{u^2}{r^2} \ \emph{d}v_{\hn} +  \frac{\gamma_{N}(\lambda)(\gamma_{N}(\lambda)+1)}{2} \int_{\hn} g(r)\, u^2  \ \emph{d}v_{\hn} \\
& +  \frac{(N-1+\gamma_{N}(\lambda))(N-3-\gamma_{N}(\lambda))}{4} \int_{\hn}  \frac{u^2}{\sinh^2 r} \ \emph{d}v_{\hn}\,,
\end{aligned}
\end{equation}
where $\gamma_{N}(\lambda):=\sqrt{(N-1)^2-4\lambda}$ and $g$ is defined by
\begin{equation}\label{gdef}g(r) = \frac{r \coth r - 1}{r^2} >0.\end{equation} The function $g$ is strictly decreasing and satisfies
$$g(r) \sim \frac{1}{3} \ \mbox{ as } \  r \rightarrow 0^+ \quad \mbox{ and }  \quad g(r) \sim \frac{1}{r} \  \mbox{ as }  \  r \rightarrow +\infty\,.$$ Besides, the operator $ -\Delta_{\hn} -\lambda -V_{\lambda}(r)$ with the positive potential $V_\lambda$ being given by
\begin{equation}\begin{aligned}\label{potential}
V_{\lambda}(r)&:= \frac{(\gamma_{N}(\lambda)+1)^2}{4}  \frac{1}{r^2}+ \frac{\gamma_{N}(\lambda)(\gamma_{N}(\lambda)+1)}{2} g(r)\\ &+ \frac{(N-1+\gamma_{N}(\lambda))(N-3-\gamma_{N}(\lambda))}{4}  \frac{1}{\sinh^2 r}
\end{aligned}
\end{equation}
 is critical in $\hn \setminus \{ x_0 \}$ in the sense that the inequality

$$
\int_{\hn} |\nabla_{\hn} u|^2 \ \emph{d}v_{\hn}-\lambda \int_{\hn}  u^2\, \emph{d}v_{\hn} \geq \int_{\hn} V u^2 \ \emph{d}v_{\hn} \quad\forall u \in C_{c}^{\infty} (\hn \setminus \{ x_0 \})
$$
is not valid for all $u \in C_{c}^{\infty} (\hn \setminus \{ x_0 \})$ given any $V \gneqq V_{\lambda} $.

\end{thm}


\begin{rem}[Asymptotics of $V_\lambda(r)$]\label{asym}
We investigate here the behavior of $V_{\lambda}$ at zero and at infinity. For any $ \lambda \leq  \lambda_{1}(\hn)$, there holds
$$
V_{\lambda}(r) = \left(\frac{N-2}{2} \right)^2  \frac{1}{r^2} +R_N(\lambda)+o(r)\quad \text{as } r \rightarrow 0^+ \,,
  $$
where $R_N(\lambda):=(\lambda_{1}(\hn)-\lambda)+\frac{2}{3} \sqrt{\lambda_{1}(\hn)-\lambda}-\frac{(N-1)(N-3)}{12}$ and the map $(-\infty, \lambda_{1}(\hn)]\ni \lambda \mapsto R_N(\lambda)$ is decreasing. Hence, among the weights $V_{\lambda}$,  $V_{\lambda_{1}(\hn)}$, is the \lq \lq smallest" near the origin. On the other hand, if we consider the weights $\overline V_{\lambda}(r)=(\lambda_{1}(\hn)-\lambda)+V_{ \lambda_{1}(\hn)}(r)$ as defined in the Introduction, we have that
$$
\overline V_{\lambda}(r) = \left(\frac{N-2}{2} \right)^2  \frac{1}{r^2} +\overline R_N(\lambda)+o(r)\quad \text{as } r \rightarrow 0^+ \,,
  $$
  where $\overline R_N(\lambda):=(\lambda_{1}(\hn)-\lambda)-\frac{(N-1)(N-3)}{12}$. Since $\overline R_N(\lambda)< R_N(\lambda)$ for any $\lambda<\lambda_{1}(\hn)$,  we conclude that $V_{\lambda}$ is  larger than $\overline V_{\lambda}$ near the origin.

We also note that when $N=2$ the first term in the above expansion of $V_{\lambda}$ vanishes, furthermore $ \lambda_{1}(\hnn)=\frac{1}{4}$ and we have
   $$
V_{\lambda}(r) =\frac{(1+\sqrt{1-4\lambda})(1+3\sqrt{1-4\lambda})}{12} +o(r)\quad \text{as } r \rightarrow 0^+ \,.
  $$
Let us turn to the asymptotic behavior at infinity. For any $N\geq 2$, there holds
$$
V_{\lambda}(r) \sim \frac{\gamma_{N}(\lambda)(\gamma_{N}(\lambda)+1)}{2 r}   \quad \text{if }\lambda< \lambda_{1}(\hn)\quad  \text{and}\quad V_{ \lambda_{1}(\hn)}(r) \sim   \frac{1}{ 4 r^2}\quad   \text{as } r \rightarrow +\infty \, ,
  $$
  while
 $$
\overline V_{\lambda}(r) \sim(\lambda_{1}(\hn)-\lambda)\quad   \text{as } r \rightarrow +\infty \,.
  $$
Hence, for $\lambda<\lambda_{1}(\hn)$, $\overline V_{\lambda}$ is larger than $ V_{\lambda}$ near infinity.
  \end{rem}

  The above difference in the behavior at infinity of $V_{\lambda}(r)$ between $\lambda =\lambda_1$ and $\lambda<\lambda_1$   might be related to a well known phenomenon for the Euclidean Laplacian, where $\lambda_1(\mathbb{R}^N) = 0$. The Hardy weight $\frac{1}{|x|^2}$ is at the borderline of short/long range potentials at infinity for $-\Delta$ in $\rn$. In particular,  the potential $(1+|x|)^{-\alpha}$ is a small perturbation of  the $-\Delta$ in $\rn$ for $N\geq 3$  if and only if $\alpha>2$ (see for example  the discussion in \cite[Example~1.1]{pinch}). On the other hand, for $\lambda<0$  the potential $(1+|x|)^{-\alpha}$ is a small perturbation of  $-\Delta -\lambda$ in $\rn$  if and only if $\alpha>1$. We do not claim that $\frac{1}{r}$ for $\lambda < \lambda_1(\hn)$ is a border line potential in the hyperbolic setting, however it would be interesting to further investigate the (sharp) borderline behavior of the potential at infinity in $\hn$ for  $\lambda < \lambda_1(\hn)$.


  \medskip

  In the following we highlight some remarkable inequalities derived from Theorem \ref{improved-hardy} by making specific choices of the parameters involved. The basic idea behind our choices is either to maximize the constant in front of the $L^2$-term, namely to maximize the gain at infinity, or to maximize the constant in front of the classical Hardy weight $\frac{1}{r^2},$ namely to maximize the gain at  the origin.

  \medskip

  The maximum value of the constant in front of the $L^2$-term is clearly achieved for $\lambda=\lambda_{1}(\hn)$. Since $\gamma_{N}(\lambda_{1}(\hn))=0$, for this choice of $\lambda$ the constant in front of the function $g(r)$ in \eqref{potential} vanishes so that $V_{\lambda}$ coincides with the potential in \eqref{V_lambda1} which was  introduced in \cite{AK, BGG}. Therefore, \eqref{improved-poinc-lambda} includes the \emph{sharp} Poincar\'e inequality of \cite[Theorem 2.1]{BGG}.

  \medskip

  Next we consider the constant $\frac{(\gamma_{N}(\lambda)+1)^2}{4}$ in front of the weight $\frac{1}{r^2}$ in \eqref{potential}. For $N\geq 3$ its value cannot exceed the Hardy constant and its maximum is achieved for $\gamma_{N}(N-2)=N-3$, namely for $\lambda=N-2$. For this choice of $\lambda$ the coefficient in front of the term involving $ \frac{1}{\sinh^2 r}$ vanishes and Theorem \ref{improved-hardy} yields the following \emph{sharp} Hardy inequality on $\hn.$


  \begin{cor}\label{improved-hardy_1}
Let $N \geq 3.$ For all $u \in C_{c}^{\infty} (\hn \setminus \{ x_0 \})$ there holds
\begin{align}\label{improved-hardy-eq}
\int_{\hn} |\nabla_{\hn} u|^2 \ \emph{d}v_{\hn} & \geq \left(\frac{N-2}{2} \right)^2 \int_{\hn} \frac{u^2}{r^2} \ \emph{d}v_{\hn} +  (N-2) \int_{\hn} u^2 \ \emph{d}v_{\hn} \notag \\
& + \frac{(N-2)(N-3)}{2} \int_{\hn} g(r) \, u^2 \ \emph{d}v_{\hn}\,,
\end{align}
where $g(r)$ is as given in \eqref{gdef}. Besides, the operator
$$
-\Delta_{\hn} - \left(\frac{N-2}{2} \right)^2 \,\frac{1}{r^2}  - (N-2) - \frac{(N-2)(N-3)}{2}\, g(r)
$$
is critical in $\hn \setminus \{ x_0 \}$ in the sense that the inequality

$$
\int_{\hn} |\nabla_{\hn} u|^2 \ \emph{d}v_{\hn} \geq \int_{\hn} V u^2 \ \emph{d}v_{\hn} \quad \forall u \in C_{c}^{\infty} (\hn \setminus \{ x_0 \})
$$
is not valid for all $u \in C_{c}^{\infty} (\hn \setminus \{ x_0 \})$ given any $V \gneqq\left(\frac{N-2}{2} \right)^2 \,\frac{1}{r^2} + (N-2) + \frac{(N-2)(N-3)}{2}\, g(r)$.

Moreover, the constant $\left(\frac{N-2}{2} \right)^2$ is sharp by construction, while the constants $(N-2)$ and $\frac{(N-2)(N-3)}{2}$ are \lq \lq jointly" sharp in the sense that that no inequality of the form

\[\begin{aligned}
\int_{\hn}  |\nabla_{\hn} u|^2 \ \emph{d}v_{\hn} &\geq \left(\frac{N-2}{2} \right)^2 \, \int_{\hn} \frac{u^2}{r^2} \ \emph{d}v_{\hn}\\ &+ \frac{(N-2)(N-3)}{2} \int_{\hn} g(r) \,u^2 \ \emph{d}v_{\hn}  +  c \int_{\hn} u^2 \ \emph{d}v_{\hn}
\end{aligned}\]
holds for all $u \in C_{c}^{\infty} (\hn\setminus \{ x_0 \})$ when $c > N-2$. Similarly, no inequality of  the form

$$
\int_{\hn} |\nabla_{\hn} u|^2 \ \emph{d}v_{\hn} \geq \left(\frac{N-2}{2} \right)^2 \, \int_{\hn} \frac{u^2}{r^2} \ \emph{d}v_{\hn} + (N-2) \int_{\hn} u^2 \ \emph{d}v_{\hn} + c \int_{\hn} g(r)\, u^2 \ \emph{d}v_{\hn}
$$
holds for all $u \in C_{c}^{\infty} (\hn\setminus \{ x_0 \})$ when $c > \frac{(N-2)(N-3)}{4}.$
\end{cor}

\begin{rem}\label{yang}
Several contributions are available about Hardy inequality on the hyperbolic space see e.g. \cite{Carron, Dambrosio, Kombe1, Kombe2, Yang}. Yet, its improvements and related criticality issues still present open problems. In \cite{Yang}, for $N \geq 3$, the authors show that
\begin{equation}\label{best}
\int_{\hn} |\nabla_{\hn} u|^2 \ \dvh \geq \frac{(N-2)^2}{4} \int_{\hn} \frac{u^2}{r^2} \ \dvh + C_N \int_{\hn} u^2 \ \dvh \qquad u\in C_{c}^{\infty} (\hn),
\end{equation}
where $C_N \geq \frac{N-1}{4}.$ The explicit value of $C_N$ is not known and there is no information whether we can add more remainder terms in R.H.S of the inequality. See also \cite[Theorem 3.1]{Kombe1} for a Euclidean $L^2$-type improvement of inequality \eqref{hardy}.

\medskip

More recently, a new type of improved Hardy inequality has been proved in \cite[Theorem~4.1]{AKR}. In terms of  the function $g(r)$ defined in Theorem~\ref{improved-hardy}, the inequality in \cite[Theorem~4.1]{AKR} reads:
\begin{equation}\label{eq_AKR}
\int_{\hn} |\nabla_{\hn} u|^2\, {\rm d}v_{\hn} \geq  \left(\frac{N-2}{2}\right)^2 \int_{\hn} \frac{u^2}{r^2} \, {\rm d}v_{\hn} + \frac{(N-1)(N-2)}{2} \int_{\hn} g(r)\, u^2 \, {\rm d}v_{\hn},
\end{equation}
for $N\geq 3$ and all $u \in C_{c}^{\infty}(\hn)$. By noting that $g(r) \leq \frac{1}{3}$ for every $r>0$, it is readily deduced that inequality \eqref{eq_AKR} follows from our inequality \eqref{improved-hardy-eq}. Hence in particular inequality \eqref{improved-hardy-eq} is stronger than that of \eqref{eq_AKR} proved in \cite[Theorem~4.1]{AKR}.

\par\medskip
For what remarked above, although we do not have a proof of optimality of the constant given in Corollary~\ref{improved-hardy_1} for the $L^2$-remainder term, our result improves the existing ones. See Theorems \ref{maintheorem1} and \ref{maintheorem2} below and the related discussion for further interesting consequences of Corollary~\ref{improved-hardy_1}.
\end{rem}

\medskip

\begin{rem}\label{sign}
Corollary \ref{improved-hardy_1} follows by Theorem \ref{improved-hardy} by taking $\lambda=N-2$ in the potential given in \eqref{potential}, so that the coefficient in front of $\frac{1}{r^2}$ assumes its maximum value, meanwhile the coefficient in front of $g(r)$ is positive, while the last term in the expression of the weight $V_{\lambda}$ vanishes. Besides, when $N-2<\lambda\leq\lambda_{1}(\hn)$, then $\gamma_{N}(N-2)<N-3$ and all the coefficients in the definition \eqref{potential} are nonnegative, and even positive if $\lambda\neq \lambda_{1}(\hn)$. Instead, when $\lambda<N-2$, the coefficient in front of $\frac{1}{r^2}$ in \eqref{potential} still increases but since $\gamma_{N}(N-2)>N-3$ the coefficient in front of $\frac{1}{\sinh^2 r}$ becomes negative. However, since $\frac{1}{r^2}  > \frac{1}{\sinh^2 r} $ for $r>0$, $V_{\lambda}$ is still positive, indeed we have
 $$V_{\lambda}(r) \geq \left(\frac{N-2}{2} \right)^2\,\frac{1}{\sinh^2 r} +\frac{\gamma_{N}(\lambda)(\gamma_{N}(\lambda)+1)}{2} g(r)\,. $$
As already explained in Remark \ref{asym}, these weights become larger and larger near the origin as $\lambda$ decreases.

\end{rem}

\par
  \bigskip

  \par
  Going on with our analysis of consequences of Theorem \ref{improved-hardy}, we focus on the case $N=2$ that was not studied in \cite{BGG}. Taking $N=2$ in \eqref{potential}, for any $\lambda \leq 1/4$, we get  $$V_{\lambda}(r):= \frac{(\sqrt{1-4\lambda}+1)^2}{4} \left( \frac{1}{r^2}-\frac{1}{\sinh^2 r}\right)+
  \frac{(\sqrt{1-4\lambda})(\sqrt{1-4\lambda}+1)}{2} g(r) \,.$$
In particular, for $\lambda=\lambda_{1}(\hnn)=\frac{1}{4}$, Theorem \ref{improved-hardy} yields the following \emph{sharp} improved Poincar\'e inequality:

\begin{cor}\label{critical2}
For all $u \in C_{c}^{\infty}(\hnn\setminus \{ x_0 \})$ there holds :
\begin{equation}\label{2dimension}
\int_{\hnn} |\nabla_{\hnn} u|^2 \ \emph{d}v_{\hnn} - \frac{1}{4} \int_{\hnn} u^2 \ \emph{d}v_{\hnn} \geq \frac{1}{4}
 \int_{\hnn} \left( \frac{1}{r^2} - \frac{1}{\sinh^2 r}\right) u^2 \ \emph{d}v_{\hnn} \,.
\end{equation}
Moreover, the operator, $ -\Delta_{\hnn} - \frac{1}{4} - \frac{1}{4} \, \left( \frac{1}{r^2} - \frac{1}{\sinh^2 r}\right)$ is critical in $\hnn \setminus \{ x_{0} \}$, i.e. the inequality
\[
\int_{\hnn} |\nabla_{\hnn} u|^2 \ \emph{d}v_{\hnn} - \frac{1}{4} \int_{\hnn} u^2 \ \emph{d}v_{\hnn} \geq \int_{\hnn} W u^2 \ \emph{d}v_{\hnn}\quad \forall\, u \in C_{c}^{\infty}(\hnn \setminus \{ x_{0}\})
\]
is not valid for any $W \gneqq  \frac{1}{4} \, \left( \frac{1}{r^2} - \frac{1}{\sinh^2 r}\right).$  \\
In particular, all the constants in \eqref{2dimension} are sharp. In particular, no inequality of the form
\[
\int_{\hnn} |\nabla_{\hnn} u|^2 \ \emph{d}v_{\hnn} - \frac{1}{4} \int_{\hnn} u^2 \ \emph{d}v_{\hnn}
\geq  C \int_{\hnn} \left( \frac{1}{r^2} - \frac{1}{\sinh^2 r}\right)  u^2 \ \emph{d}v_{\hnn}\quad \forall\, u \in C_{c}^{\infty}(\hnn \setminus \{ x_{0}\})
\]
holds when $C > \frac{1}{4}$.

\end{cor}
%


In the next results we change our point of view, taking the attitude of fixing the Hardy weight for the
operator $-\Delta_{\hn}-\lambda$ to be $\frac{I}{r^2}$ for some $I>0$ and investigating the properties of the best constant. In other words, for any $0 \leq \lambda \leq \lambda_1(\hn)$, we study the infimum problem \eqref{quotient} which also reads

\begin{equation}\label{PH}
\int_{\hn} |\nabla_{\hn} u|^2 \ \mathrm{d} v_{\hn}-\lambda \int_{\hn}\, u^2 \ \mathrm{d}v_{\hn}  \geq  I(\lambda) \int_{\hn} \frac{u^2}{r^2}\,  \ \mathrm{d}v_{\hnn}
\quad \forall\, u \in C_{c}^{\infty}(\hnn\setminus \{ x_0 \})\,.
\end{equation}
We have already remarked that $I(0)=\left(\frac{N-2}{2} \right)^2$ and $I(\lambda_1(\hn))=\frac{1}{4}$. Since the map $\lambda \mapsto I(\lambda)$ is non increasing and concave, hence continuous, the following number is well-defined
\begin{equation}\label{hat}
  \hat{\lambda}^N  :=\max\left\{\lambda \in[0,\lambda_1(\hn)]\,:\,I(\lambda)=\left(\frac{N-2}{2} \right)^2\right\}\,.
  \end{equation}
Namely, $\hat{\lambda}^N$ is the best constant in front of the $L^2$-term such that \eqref{best} holds.
\medskip

When $N=3$, the Hardy constant and the Poincar\'e constant are both equal to $\frac{1}{4}$. Hence, from
Corollary \ref{improved-hardy_1} (or \cite[Theorem 2.1]{BGG}, see also \cite[Remark 2.2]{BGG}) it follows
that $\hat{\lambda}^3=\lambda_1({\mathbb{H}^{3}})=1$ and the following statement holds:

\begin{thm}\label{maintheorem1}
Let $N = 3$. For any $\lambda \leq 1$, $I(\lambda)= \frac{1}{4}$ and the infimum in \eqref{quotient} is not achieved, i.e. the inequality in \eqref{PH} is strict for $u\neq 0$.
\end{thm}

For higher dimensions the situation is more complicated and we have the following :

\begin{thm}\label{maintheorem2}
Let   $N > 3$ and let $\hat{\lambda}^N$ be as defined in \eqref{hat}. Then,
$$
N-2 \leq \hat{\lambda}^N <  \min\left\{\lambda_1(\hn), N-2+\frac{(N-2)(N-3)}{6}\right\}
$$
and the following three cases occur:
\begin{itemize}
\item[$(i)$] for $0 \leq \lambda \leq \hat{\lambda}^N$, $I(\lambda)=\left(\frac{N-2}{2} \right)^2$ and the infimum in \eqref{quotient} is not achieved, i.e. the inequality in \eqref{PH} is strict for $u\neq 0$;

\medskip

\item[$(ii)$] for $\hat{\lambda}^N< \lambda < \lambda_{1}(\hn)$, $\left(\frac{1+2\sqrt{\lambda_{1}(\hn)
-\lambda}}{2}\right)^2<I(\lambda)<\left(\frac{N-2}{2} \right)^2$ and the infimum in \eqref{quotient} is achieved by a unique (up to a change of sign) positive function $u\in H^1(\hn)$; in particular,  the corresponding operator is critical.

\medskip

\item[$(iii)$] for $\lambda = \lambda_{1}(\hn)$, $I(\lambda_{1}(\hn))=\frac{1}{4}$ and the infimum in \eqref{quotient} is not achieved, i.e. the inequality in \eqref{PH} is strict for $u\neq 0$.
\end{itemize}
\end{thm}

\begin{rem}\label{upper_bound_rem}
Note that for $N\geq 6$ there holds $$\min\left\{\lambda_1(\hn), N-2+\frac{(N-2)(N-3)}{6}\right\}=N-2+\frac{(N-2)(N-3)}{6}\,.$$
\end{rem}

\textbf{ Open problems related to Theorem \ref{maintheorem2}}: \par
\medskip

- Theorem \ref{maintheorem2} does not give the explicit value of $\hat{\lambda}^N$. The strict inequality in the lower bound provided for $I(\lambda)$ in
the statement $(ii)$ of Theorem~\ref{maintheorem2} and the inequality

$$
\left( \frac{1 + 2 \sqrt{\lambda_1(\hn) - \hat\lambda^N}}{2}  \right)^2 \leq \left( \frac{1 + 2 \sqrt{\lambda_1(\hn) - (N-2)}}{2}  \right)^2 = \frac{(N-2)^2}{4}
$$

suggest the conjecture that $\hat{\lambda}^N>N-2$ but we do not have a proof of this fact; \par
\medskip

- By Theorem \ref{maintheorem2} it is readily deduced that the operator
$$
-\Delta_{\hn} - \lambda-  \, \frac{I(\lambda)}{r^2}\,,
$$
is critical for $\hat{\lambda}^N< \lambda <\lambda_{1}(\hn)$ while it is subcritical for $0 \leq \lambda < \hat{\lambda}^N$ and for $ \lambda =\lambda_{1}(\hn)$ (subcriticality for $\lambda = \lambda_1(\hn)$ comes from \cite[Theorem~2.1]{BGG}, namely from the existence of the weight \eqref{V_lambda1}). We do not have a proof of the subcriticality/criticality of the operator when $\lambda = \hat{\lambda}^N$.


\subsection{A second  Hardy-type inequality on the hyperbolic space and a further upper bound on $\hat{\lambda}^N$}
 Now we  study a different Hardy-type inequality on the hyperbolic space which resembles the classical Hardy one near the pole $x_0.$ It is quite natural to consider a Hardy weight  related to the defining function of $\hn$ as a model manifold, namely to the quantity $\sinh r$,  that behaves like $ r $ near pole and decays exponentially near infinity. We shall produce an \emph{optimal} Hardy-type inequality, in particular we have the following result.

  \begin{thm}\label{improved-hardy_different}
Let $N \geq 3.$ For all $u \in C_{c}^{\infty} (\hn \setminus \{ x_0 \})$ there holds
\begin{align}\label{improved-hardy-eq2}
\int_{\hn} |\nabla_{\hn} u|^2 \ \emph{d}v_{\hn} & \geq \left(\frac{N-2}{2} \right)^2 \int_{\hn} \frac{u^2}{\sinh^2 r} \ \emph{d}v_{\hn} +  \frac{1}{4}
\int_{\hn} \frac{u^2}{\sinh^2 r (\log (\tanh \frac{r}{2}))^2} \, \emph{d}v_{\hn} \notag \\
& + \frac{N(N-2)}{4} \int_{\hn}  \, u^2 \ \emph{d}v_{\hn}\,.
\end{align}

 Besides, the operator

\begin{equation}\label{critical_hardy_type}
-\Delta_{\hn} - \left(\frac{N-2}{2} \right)^2 \,\frac{1}{\sinh^2 r} \, - \, \frac{1}{4\sinh^2 r (\log (\tanh \frac{r}{2}))^2}   - \frac{N(N-2)}{4}
\end{equation}
is critical in $\hn \setminus \{ x_0 \}$ in the sense described in Theorem~\ref{improved-hardy}. Moreover, the constant $\frac{N(N-2)}{4}$ is
sharp in the sense that no inequality of the form

$$
\int_{\hn} |\nabla_{\hn} u|^2 \ {\rm d}v_{\hn}  \geq \left(\frac{N-2}{2} \right)^2 \int_{\hn} \frac{u^2}{\sinh^2 r} \, {\rm d}v_{\hn} + c \int_{\hn} u^2 \, {\rm d}v_{\hn}
$$
holds for all $u \in C_c^{\infty}(\hn \setminus \{ x_0 \})$ when $c > \frac{N(N-2)}{4}.$

\end{thm}

As an immediate consequence, we get the following result.

\begin{cor}
Let   $N \ge 3$ and let $\hat{\lambda}^N$ be as defined in \eqref{hat}. Then:
$$
\hat{\lambda}^N  \leq \frac{N(N-2)}{4}.
$$
\end{cor}

\begin{rem}
 One sees that $\frac{N(N-2)}{4} <  N-2+ \frac{(N-2)(N+3)}{6} $ if and only if $N < 6,$ whereas of course $\frac{N(N-2)}{4}<\lambda_1(\hn)$ for all $N$.  Hence, the above corollary provides
a better  upper bound  for $\hat{\lambda}^N$ than in Remark~\ref{upper_bound_rem} for dimension $N=4, N=5$.
\end{rem}


\section{Heisenberg-Pauli-Weyl uncertainty principle for the shifted Laplacian in the hyperbolic space}\label{HPW}

In this section we state some quantitative versions of Heisenberg-Pauli-Weyl uncertainty principle (HPW) that can be derived from Theorem \ref{improved-hardy}. Firstly we recall that HPW principle in the hyperbolic setting reads
\begin{equation}\label{eq_HPW_0}
\left( \int_{\hn} |\nabla_{\hn} u|^2 \, {\rm d}v_{\hn} \right) \left( \int_{\hn}\, r^2 u^2 \,  {\rm d}v_{\hn} \right)
\geq \frac{N^2}{4}   \left(  \int_{\hn}  u^2 \, {\rm d }v_{\hn} \right)^2
\end{equation}
for all $u \in C_{c}^{\infty} (\hn \setminus \{ x_0 \})$. The constant $\frac{N^2}{4} $ is sharp and the equality is not attained for $u\neq 0$. We refer to \cite{AKR} for a description of a complete scenario of HPW principle on complete Riemannian manifolds.\par
For what remarked in the Introduction, one may wonder what happens if we replace the first term in \eqref{eq_HPW_0} with the quadratic form associated to the family of nonnegative operators $P_{\lambda}=-\Delta_{\hn} -\lambda$ with $\lambda \leq   \lambda_{1}(\hn)$. Clearly, the related best constant must be nonincreasing with respect to $\lambda$. In Corollary \ref{HPW2} below we provide a lower bound for the constant which reflects this monotonicity property. Indeed, by combining Theorem \ref{improved-hardy} with Cauchy-Schwarz inequality, one immediately obtains the following quantitative version of HPW principle in $\hn$:

\begin{cor}\label{HPW2}
 Let $N \geq 2.$ For all $u \in C_{c}^{\infty} (\hn \setminus \{ x_0 \})$,

 \begin{itemize}

 \item if $\lambda \leq N-2$ there holds

 $$
 \left( \int_{\hn} \left( |\nabla_{\hn} u|^2 - \lambda u^2 \right) {\rm d}v_{\hn} \right)\, \left( \int_{\hn} r^2 u^2  \,  {\rm d}v_{\hn} \right)
\geq \left(\frac{N-2}{2} \right)^2  \left(  \int_{\hn}  u^2 \, {\rm d}v_{\hn} \right)^2\,;
 $$

 \item if $N-2 <  \lambda \leq \lambda_1(\hn) $ there holds

$$
 \left( \int_{\hn} \left( |\nabla_{\hn} u|^2 - \lambda u^2 \right) {\rm d}v_{\hn} \right)\, \left( \int_{\hn} r^2 u^2  \,  {\rm d}v_{\hn} \right)
\geq  \frac{(\gamma_N(\lambda) + 1)^2}{4}   \left(  \int_{\hn}  u^2 \, {\rm d}v_{\hn} \right)^2 $$

 where
$\gamma_{N}(\lambda)$ is as defined in Theorem \ref{improved-hardy}.
  \end{itemize}
 \end{cor}
 Notice that when $\lambda=\lambda_1(\hn)$ and $N\geq 3$, Corollary \ref{HPW2} was already known from \cite{BAGG}. However, since the map $\lambda\in [N-2, \lambda_1(\hn)] \mapsto  \frac{(\gamma_N(\lambda) + 1)^2}{4} $ decreases from $\left(\frac{N-2}{2} \right)^2$ to $\frac{1}{4}$, the validity of the HPW principle for $\lambda=\lambda_1(\hn)$ does not yield the HPW principle for $\lambda<\lambda_1(\hn)$.
 \\

 When $N\geq 3$ and $\lambda>N-2$, by repeating the same argument of Corollary \ref{HPW2}, but with a finer exploitation of Theorem \ref{improved-hardy}, we derive the following improved HPW principle:

 \begin{cor}\label{HPW3}
 Let $N \geq 3$ and $N-2 <  \lambda \leq \lambda_1(\hn)$. For all $u \in C_{c}^{\infty} (\hn \setminus \{ x_0 \})$, there holds

\[
\begin{aligned}
 &\left( \int_{\hn} \left( |\nabla_{\hn} u|^2 - \lambda u^2 \right) {\rm d}v_{\hn} \right)\, \left( \int_{\hn} r^2 u^2  \,  {\rm d}v_{\hn} \right)\\
&\geq  \frac{(\gamma_N(\lambda) + 1)^2}{4}   \left(  \int_{\hn}  u^2 \, \emph{d}v_{\hn} \right)^2
 + \left( \int_{\hn} r^2 u^2  \,  \emph{d}v_{\hn} \right)\times \\ &\times\left(\int_{\hn} \left({ \frac{\gamma_{N}(\lambda)(\gamma_{N}(\lambda)+1)}{2} \, g(r)+  \frac{(N-1+\gamma_{N}(\lambda))(N-3-\gamma_{N}(\lambda)}{4\,\sinh^2 r}}  \right)\, u^2 \emph{d}v_{\hn}\right),
\end{aligned}
\]
 where
 $g(r)>0$ and $0\leq \gamma_{N}(\lambda)< N-3$ are as defined in Theorem \ref{improved-hardy}.
 \end{cor}
 \par

 The proof of Corollary \ref{HPW3} is similar to that of Corollary~\ref{HPW_1} below,  hence we omit it. \par
 Coming back to Corollary \ref{HPW2}, for $\lambda= 0$ it yields a weaker inequality than \eqref{eq_HPW_0}. Nevertheless, in the spirit of Corollary \ref{HPW3}, a finer exploitation of Theorem \ref{improved-hardy} yields a more powerful quantitative HPW principle in $\hn$:

\begin{cor}\label{HPW_1}
Let $N \geq 2.$ For all $u \in C_{c}^{\infty} (\hn \setminus \{ x_0 \})$ there holds

\begin{equation}\label{eq_HPW_1}
\left( \int_{\hn} |\nabla_{\hn} u|^2 \, \emph{d}v_{\hn} \right) \left( \int_{\hn}  \alpha(r)\, r^2 u^2 \,  \emph{d}v_{\hn} \right)
\geq \frac{N^2}{4}   \left(  \int_{\hn}  u^2 \, \emph{d}v_{\hn} \right)^2,
\end{equation}
with
$$
\alpha(r) = \frac{1}{1 + \frac{2(N-1)}{N^2}\,r^2 \left( Ng(r)- \frac{2}{\sinh^2 r}\right) }>0
$$
and $g(r)>0$ as defined in Theorem \ref{improved-hardy}.  Moreover, there exists  $\overline{R}=\overline R(N) > 0$ such that

\begin{equation}\label{function_alpha}
\alpha(r) \geq  1 \quad \forall \ 0 \leq r \leq \overline {R} \quad \emph{and} \quad  \alpha(r) <  1 \quad \forall \ r > \overline {R}.
\end{equation}
\end{cor}
It is worth noting that, even if we do not know whether the inequality in Corollary \ref{HPW_1} is sharp, the behavior of the function $\alpha(r)$ outlined in \eqref{function_alpha} indicates that inequality \eqref{eq_HPW_1} does not follow from \eqref{eq_HPW_0}. Besides, inequality \eqref{eq_HPW_1} becomes more powerful than inequality \eqref{eq_HPW_0} for functions having support outside the ball $B_{\overline R}(0)$.

\begin{proof}[Proof of Corollary~\ref{HPW_1}]
It suffices to notice that, by Cauchy-Schwarz inequality and Theorem~\ref{improved-hardy} for $\lambda = 0$ :
\begin{align*}
\int_{\hn} u^2\, {\rm d}v_{\hn} &= \int_{\hn} \frac{|u|}{\sqrt {V_0(r)}} |u| \sqrt{V_0(r)} \, {\rm d}v_{\hn}\\
& \leq \left( \int_{\hn} \frac{u^2}{V_0(r)}\, {\rm d}v_{\hn} \right)^{\frac{1}{2}} \left( \int_{\hn} |u|^2 V_0(r) \, {\rm d}v_{\hn} \right)^{\frac{1}{2}} \\
&  \leq \left( \int_{\hn} \frac{u^2}{V_0(r)}\, {\rm d}v_{\hn} \right)^{\frac{1}{2}} \left( \int_{\hn} |\nabla_{\hn} u|^2\, {\rm d}v_{\hn} \right)^{\frac{1}{2}},
\end{align*}
where $V_{0}(r) = \frac{N^2}{4} \,\left( \frac{1}{r^2} - \frac{1}{\sinh^2 r}\right)+ \frac{(N-2)^2}{4} \, \frac{1}{\sinh^2 r}+\frac{N(N-1)}{2}\, g(r) \,. $ Inserting the value
of $V_0$ in the formula above and defining $\alpha(r) = \frac{N^2}{4 V_0^2(r) r^2} $ we obtain \eqref{eq_HPW_1}. Also by rewriting  $V_0$ we obtain $\alpha(r)$ as defined in the statement
and hence, using the fact that $\alpha(r) < 1 \Leftrightarrow N g(r) > \frac{2}{\sinh^2 r},$ we obtain $\overline  R= \overline R(N) > 0$ such that  \eqref{function_alpha} holds true.
\end{proof}


\section{Improved Hardy inequalities on general Cartan-Hadamard manifolds}\label{manifolds}

In the present section we state a generalization of the improved Hardy inequality of Theorems~\ref{improved-hardy}  and \ref{improved-hardy_different} to more general manifolds
under suitable curvature assumptions. Denote by $K_R$ the sectional curvature in the radial direction of a Riemannian manifold with a pole $x_0$. We assume throughout the bound
\begin{equation}\label{curv}
K_R(x)\le -G(r(x))\le0\qquad  \forall x\in M,
\end{equation}
where $G$ is a given function and $r(x)= {\rm d}(x,x_0)$. In particular, we are assuming that $M$ is Cartan-Hadamard. We also define $\psi$ to be the solution to the Cauchy problem
\begin{equation}\label{curvmod}
\begin{cases}\psi''(r)-G(r)\psi(r)=0 \qquad r>0,
\\ \psi(0)=0,\ \psi'(0)=1.\end{cases}
\end{equation}
Clearly, by the sign assumption on $G$, $\psi$ is positive convex function, and in particular by the initial condition we have $\psi(r)\ge r$ for all $r\ge0$. One can adapt the present results to manifolds with pole being positively curved somewhere, under suitable smallness conditions.

We shall use the well-known strategy of constructing barriers using Hessian comparison and equations posed on the \emph{Riemannian model} $M_{\psi}$ associated to $\psi$ constructed above. Namely, we consider the $N$-dimensional Riemannian manifold $M_{\psi}$ admitting a pole  $x_0$,  whose metric is given in spherical coordinates by
 \begin{equation}\label{meetric}
 {\rm d}s^2 = {\rm d}r^2 + \psi^2(r) \, {\rm d}\omega^2,
 \end{equation}
 where $ {\rm d}\omega^2$ is the standard metric on the sphere $\mathbb{S}^{N-1}$. The coordinate $r$ represents the Riemannian distance from the pole $x_0,$ see e.g. \cite{RGR, PP} for further details. For Riemannian models the curvature condition in \eqref{curv} holds with an equality. Clearly, for $\psi(r)=r$ one has $M_{\psi}=\mathbb R^N$, while for $\psi(r)=\sinh r$ one has $M_{\psi}=\hn$.

 \medskip


Now we are in a position to state the counterpart of Theorem~\ref{improved-hardy} under more general curvature conditions.

\begin{thm}\label{general manifold2}
Let $N\geq 2$ and let $M$ be an $N$-dimensional Cartan-Hadamard manifold such that the curvature condition \eqref{curv} holds.  Let $\psi$ be defined in \eqref{curvmod} and let  $ \lambda \leq  \left(\frac{N-1}{2} \right)^2$.  
Then, for all $u \in C_c^{\infty}(M\setminus\{x_0\}),$ there holds

\begin{equation}\label{generalHardylamb}
 \int_{M} |\nabla_{M} u|^2 \ {\rm d}v_g \geq   \frac{(\gamma_{N}(\lambda)+1)^2}{4} \int_{M} \frac{u^2}{r^2} \ \emph{d}v_g +  \int_{M} V_{\psi}^{\lambda}\, u^2 \ {\rm d}v_g,
  \end{equation}
where $\gamma_{N}(\lambda):=\sqrt{(N-1)^2-4\lambda}$, and
$$V_{\psi}^{\lambda}: = \left[ \frac{N-1+\gamma_{N}}{2} \,\frac{\psi^{\prime \prime}}{\psi} +  \frac{\gamma_{N}(\gamma_{N}+1)}{2}  \left(\frac{r\frac{\psi^{\prime}}{\psi}-1}{r^2} \right) +  \frac{(N-1+\gamma_{N})(N-3-\gamma_{N})}{4}\left(\frac{\psi'}{\psi}\right)^2  \right]\,.$$
Furthermore, the inequality \eqref{generalHardylamb} is sharp in the sense that the operator
$$
-\Delta_{M} - \frac{(\gamma_{N}(\lambda)+1)^2}{4} \, \frac{1}{r^2} - V_{\psi}^{\lambda}
$$
is critical in $M\setminus \{ x_0 \}$ when $M$ coincides with the Riemannian model $M_{\psi}$.
\end{thm}


A special case of the above construction is the situation in which the curvature bound is simply $ K_{R}(x) \leq -c$ for some $c>0$.   In this case, it is readily checked that the solution of \eqref{curvmod} is given by $\psi(r)=\sqrt{c}\sinh(\sqrt{c} r)$. Writing \eqref{generalHardylamb} with $\psi(r)=\sqrt{c}\sinh(\sqrt{c} r)$, from Theorem \ref{general manifold2} we derive the following analogue of the improved inequality \eqref{improved-poinc-lambda} on Cartan-Hadamard manifolds having sectional curvature bounded above by a negative constant:
\begin{cor}\label{general manifold cor}
Let $N\geq 3$ and let $M$ be a Cartan-Hadamard manifold with pole $x_0$ such that $K_{R}(x) \leq -c$ for some $c>0$. Let $ \lambda \leq  \left(\frac{N-1}{2} \right)^2$. Then the following improved Hardy inequality holds
 \begin{align}\label{improved-poinc-lambda2}
 &\int_{M}|\nabla_{M} u|^2\, \emph{d}v_{g} -\lambda \,c\, \int_{M}  u^2\, \emph{d}v_{g}\notag\\ &\geq \frac{(\gamma_{N}(\lambda)+1)^2}{4} \int_{M} \frac{u^2}{r^2} \ \emph{d}v_{g} +  \frac{\gamma_{N}(\lambda)(\gamma_{N}(\lambda)+1)}{2} \int_{M} g_c(r)\, u^2  \ \emph{d}v_{g}  \\
& +  c\, \frac{(N-1+\gamma_{N}(\lambda))(N-3-\gamma_{N}(\lambda)}{4} \int_{M}  \frac{u^2}{\sinh^2 (\sqrt{c}r)} \ \emph{d}v_{g}\, \notag
\end{align}
for all $u \in C_c^{\infty}(M\setminus\{x_0\}),$ where $g_c(r):=\frac{r\sqrt{c}\coth (\sqrt{c}r)-1}{r^2}$\,.
\end{cor}

\par

Although the following is a particular case of Corollary~\ref{general manifold cor} (for $ \lambda = (N-2)$), we state explicitly this case for its special significance in improving the sharp Hardy inequality.


\begin{cor}\label{general manifold cor2}
Let $N\geq 3$ and let $M$ be a Cartan-Hadamard manifold with pole $x_0$ such that $K_{R}(x) \leq -c$ for some $c>0$. Then the following improved Hardy inequality holds
$$\begin{aligned}
 \int_{M} |\nabla_{M} u|^2 \ {\rm d}v_g &\geq  \left(\frac{N-2}{2} \right)^2 \int_{M} \frac{u^2}{r^2} \ {\rm d}v_g \\ &+ c (N-2) \int_{M} u^2 \, \ {\rm d}v_g +
\frac{(N-2)(N-3)}{2}   \int_{M} g_c(r) u^2 \, {\rm d}v_g
 \end{aligned}$$
for all $u \in C_c^{\infty}(M),$ where $g_c(r):=\frac{r\sqrt{c}\coth (\sqrt{c}r)-1}{r^2}$\,.
\end{cor}
Clearly, $g_1(r)=g(r)$ with $g(r)$ as defined in Theorem~\ref{improved-hardy}.

\begin{rem}
One can consider in an almost explicit way other classes of curvature bounds in \eqref{curv}. For example, if the manifold satisfies  the curvature bound \eqref{curv} with $G(r)\sim Cr^{2a}$ as $r\to+\infty$  for some $a >-1$, one can take $\psi(r)\sim Ae^{br^{a+1}}$, see e.g. \cite[Sect. 2.3]{GMV}.  In this case, the potential $V_\psi^\lambda$ in Theorem \ref{general manifold2} satisfies, if $a>0$,

\[
V_\psi^\lambda(r)\sim \lambda  (a+1)^2b^2r^{2a}\ \ \ \textrm{as}\ r\to+\infty.
\]

The case $a=0$ has been dealt with in the previous Corollaries. If $a\in(-1,0)$ the leading term is a pure Hardy one. The case $a<-1$ which is qualitatively Euclidean and in fact yields  a pure Hardy potential, and the case $a=-1$ which gives rise to functions $\psi$ of a different kind (see again \cite{GMV}), are left to the reader.

Moreover, if the curvature bound is written in terms of the quantity $-C(1+r^2)^{a}$ instead, for all $r$ and for an appropriate value of $C$, $\psi$ can be written explicitly, see the calculations in \cite[Appendix A]{BRM}.
\end{rem}

Our final result in this section is an analogue of Theorem \ref{improved-hardy_different} on general Cartan-Hadamard manifolds.

\begin{thm}\label{improved-hardy_general}
Let $N\geq 2$ and let $M$ be an $N$-dimensional Cartan-Hadamard manifold such that the curvature condition \eqref{curv} holds. Let $\psi$ be defined in \eqref{curvmod} and assume that $1/\psi$ is integrable at infinity. Denote
\begin{equation}\label{v}
\frac1{\Theta(r)}:=\int_r^{+\infty}\frac1{\psi(s)}\,{\rm d}s,
\end{equation}
and consider the nonnegative potential $U_\psi$
\begin{equation}\label{potential general}
U_\psi(r):=\frac{(N-2)^2}{4}\frac{[\psi'(r)]^2-1}{\psi^2(r)}+\frac{N-2}2\frac{\psi''(r)}{\psi(r)}\,.
\end{equation}

Then for all $u \in C_{c}^{\infty} (M \setminus \{ x_0 \})$ there holds
\begin{equation}\begin{aligned}\label{improved-hardy-eq2-general}
\int_{M} |\nabla_{M} u|^2 \ \emph{d}v_{g} & \geq \left(\frac{N-2}{2} \right)^2 \int_{M} \frac{u^2}{\psi^2(r)} \ \emph{d}v_{g}+ \int_{M}  U_\psi(r) u^2 \ \emph{d}v_{g}  \\
& +  \frac{1}{4}
\int_{M} \frac{\Theta^2(r)}{\psi^2(r)}u^2 \, \emph{d}v_{g} \,.
\end{aligned}\end{equation}

Moreover, \eqref{improved-hardy-eq2-general} is sharp in the sense that the operator
\begin{equation}\label{critical_hardy_type2}
-\Delta_{M} - \left(\frac{N-2}{2} \right)^2 \,\frac{1}{\psi^2(r)} \, - U_\psi(r)\,-\, \frac{\Theta^2(r)}{4\psi^2(r)}
\end{equation}
is critical in $M\setminus \{ x_0 \}$ when $M$ coincides with the Riemannian model $M_{\psi}$.
\end{thm}

\begin{rem}
	
	\begin{itemize}
\item The quantities appearing in the potential $U_\psi$ defined in \eqref{potential general} have a clear geometrical meaning: in fact,
\[
  K_{\pi,r}^{rad} = - \frac{\psi^{\prime \prime}}{\psi}\quad \text{and}  \quad H_{\pi,r}^{tan} = - \frac{(\psi^{\prime})^2 - 1}{\psi^{2}}\,,
 \]
where $K_{\pi,r}^{rad}$ (resp. $H_{\pi,r}^{tan}$) denotes sectional curvature relative to planes containing (resp. orthogonal to) the radial direction in the Riemannian model associated to $\psi$. Clearly, $U_\psi$ is nonnegative given the assumed sign condition on the curvature.
\item Theorem \ref{improved-hardy_general} shows that the results of Theorem \ref{improved-hardy_different} also hold when the radial sectional curvature satisfies $K_R\le -1$ anywhere. Of course, a variant of Theorem \ref{improved-hardy_different} can be stated by applying Theorem \ref{improved-hardy_general} when $K_R\le -c<0$ proceeding as in the proof of Corollary \ref{general manifold cor}.
\item The results of Theorem \ref{improved-hardy_general} do \it not \rm hold on general Cartan-Hadamard manifolds, because of the request that $1/\psi$ is integrable at infinity. In fact, this request amounts qualitatively to requiring that curvature is negative enough at infinity. In particular, the required condition does not hold on ${\mathbb R}^N$. In fact, it can be shown by constructing explicitly an appropriate $\psi$ (see \cite{GMV}), that for example, it is enough  that $K_R$ satisfies an upper curvature bound outside a ball in terms of the quantity $-c/r^2$, where $c>0$.
\end{itemize}
\end{rem}


\section{Weighted Hardy and Rellich inequality on the hyperbolic space}

This section is devoted to state some further applications of our Hardy inequality, namely the derivation of suitable improved weighted Hardy and Rellich inequalities. The statements should be compared with those contained in \cite{Kombe2}, here the novelty of the improvement lies in adding a remainder term involving the function $g(r)$ as defined in Theorem~\ref{improved-hardy}. Starting with the weighted Hardy inequality we have:

\begin{thm}\label{Rellich_lemma_1}
Assume that $N - 2 - 2\alpha > 0.$ For all $u \in C_{c}^{\infty}(\hn \setminus \{ x_0 \})$ there holds

\begin{align}\label{eq_Hardy_lemma_1}
\int_{\hn} \frac{|\nabla_{\hn} u|^2}{r^{2\alpha}} \, {\rm d}v_{\hn}  & \geq \frac{(N-2 - 2 \alpha)^2}{4} \int_{\hn} \frac{u^2}{r^{2 \alpha + 2}} \, {\rm d}v_{\hn}
+ (N-2) \int_{\hn} \frac{u^2}{r^{2\alpha}} \, {\rm d}v_{\hn} \notag \\
 & +  \left( \frac{(N-2)(N-3)}{2}  - (N-1) \alpha \right) \int_{\hn} \frac{g(r)}{r^{2\alpha}} u^2 \, {\rm d}v_{\hn},
\end{align}
where $g(r)$ is as defined in \eqref{gdef}. Moreover, the constant $\frac{(N-2 - 2\alpha)^2}{4} $ is sharp in the obvious sense.
\end{thm}

 \begin{rem}
We note that the coefficient in front of the last term in \eqref{eq_Hardy_lemma_1} is positive provided that $ \alpha \leq \frac{(N-2)(N-3)}{2(N-1)}$. Nevertheless, for $ \alpha > \frac{(N-2)(N-3)}{2(N-1)}$, by recalling that $g(r)\leq \frac{1}{3}$ for every $r>0$, we infer that

 $$(N-2)\int_{\hn} \frac{u^2}{r^{2\alpha}} \, {\rm d}v_{\hn} +\left( \frac{(N-2)(N-3)}{2}  - (N-1) \alpha \right) \int_{\hn} \frac{g(r)}{r^{2\alpha}} u^2 \, {\rm d}v_{\hn}$$
 $$\geq \left(\frac{(N-2)(N+3)}{6} -\frac{(N-1)\alpha}{3}\right) \int_{\hn} \frac{u^2}{r^{2\alpha}} \, {\rm d}v_{\hn}>\frac{2(N-2)}{3}\,\int_{\hn} \frac{u^2}{r^{2\alpha}} \, {\rm d}v_{\hn} \,,$$
 for $N - 2 - 2\alpha > 0.$ Hence, inequality \eqref{eq_Hardy_lemma_1} still gives an improvement of the weighted Hardy inequality. Also see (\cite[Theorem~4.2]{Yang}).
 \end{rem}

Next we state a weighted Rellich inequality
\begin{thm}\label{Rellich_lemma_2}
Let  $0 < \alpha < \frac{N-2}{2}.$ For all $u \in C_{c}^{\infty}(\hn \setminus \{ x_0 \})$ there holds:
 \begin{align}\label{eq_Rellich_lemma_2}
 \int_{\hn} \!\!\! \frac{|\Delta_{\hn}u|^2}{{r^{2 \alpha - 2}} }{\rm d}v_{\hn}& \!\geq\! \frac{(N-2 - 2 \alpha)^2 (N - 2 + 2 \alpha)^2}{16} \int_{\hn} \!\frac{u^2}{r^{2\alpha + 2}}  {\rm d}v_{\hn} \notag \\
 & \!+\! \frac{(N - 2 - 2 \alpha)(N - 2 + 2 \alpha)(N-2)}{2} \int_{\hn} \frac{u^2}{r^{2 \alpha}} \, {\rm d}v_{\hn} \notag \\
 & \!+\!\frac{(N \!-\! 2\! -\! 2\alpha)(N \!-\!2 \!+ \!2 \alpha)}{2} \!\left( \!\frac{(N\!-\!2)(N\!-\!3)}{2} \! -\! (N\!-\!1) \alpha\!\right)\! \!\int_{\hn}\!\! \frac{g(r)}{r^{2\alpha}} u^2  {\rm d}v_{\hn},
\end{align}
where $g(r)$ is defined in \eqref{gdef}. Moreover, the constant $\frac{(N-2 - 2 \alpha)^2 (N - 2 + 2 \alpha)^2}{16}$ is sharp in the obvious sense.
\end{thm}

 \par

 Taking $\alpha=1$ in \eqref{eq_Rellich_lemma_2}, one has the following improved Rellich inequality:

 \begin{cor}\label{Rellich_hyperbolic}
 Let $N>4.$ For all $u \in C^{\infty}_{c}(\hn \setminus \{ x_0 \})$ there holds:
 \begin{align}\label{eq_Rellich}
 \int_{\hn} (\Delta_{\hn}u)^2 \, {\rm d}v_{\hn} & \geq \frac{N^2 (N-4)^2}{16} \int_{\hn} \frac{u^2}{r^4} \, {\rm d}v_{\hn} + \frac{N(N-2)(N-4)}{2} \int_{\hn} \frac{u^2}{r^2} \, {\rm d}v_{\hn} \notag \\
&  + \frac{ N(N-4)(N^2 - 7N + 8)}{4} \int_{\hn} g(r)\,\frac{u^2}{r^2} \, {\rm d}v_{\hn},
  \end{align}
  where $g(r)$ is defined in \eqref{gdef}.
 Moreover, the constant $\frac{N^2(N -4)^2}{16}$ is sharp in the obvious sense.

 \end{cor}


\section{Proof of Theorems~\ref{improved-hardy} and \ref{improved-hardy_different}}\label{proof_thm}

We begin the proof by establishing the following lemma.

\begin{lem}\label{main-lemma}
Let $N\geq 2$ and let $\Psi(r) := r^{\alpha}(\sinh r)^{\beta},$ where $\alpha$ and $\beta$ are real parameters. Then $\Psi$ satisfies the following equation
\begin{align}\label{main-lemma-eq_1}
-\Delta_{\hn} \Psi  & = - \alpha(\alpha -1) \frac{\Psi}{r^2} - (2 \alpha \beta + (N-1) \alpha) \frac{\coth r}{r} \Psi - \beta( \beta + N -2 ) \frac{\Psi}{\sinh^{2} r} \notag \\
& - (\beta^2 + (N-1) \beta) \Psi \quad \text{in } \hn \setminus\{x_0\}\,.
\end{align}
Moreover if we assume that $\alpha = -\left( \beta + \frac{N-2}{2} \right),$ then \eqref{main-lemma-eq_1} yields
\begin{align}\label{main-lemma-implication}
-\Delta_{\hn} \Psi  & =  A(\beta) \frac{\Psi}{r^2} + B(\beta) \frac{\coth r}{r} \Psi + C(\beta) \frac{\Psi}{\sinh^2 r} + D(\beta) \Psi \quad \text{in } \hn \setminus\{x_0\}\,,
\end{align}
where $A(\beta) = -\left( \beta + [(N-2)/2] \right) \left( \beta + (N/2) \right),$ $B(\beta) = \left( \beta + [(N-2)/2] \right) \left( 2 \beta + N-1 \right)$, $C(\beta) = -\beta \left( \beta + N-2 \right)$
 and $D(\beta) = - \beta \left( \beta + N -1 \right).$ In particular, since
 $$
 A(\beta) + B(\beta) + C(\beta) = \left(\frac{N-2}{2} \right)^2\,,
 $$
 \eqref{main-lemma-implication} yields
$$
-\Delta_{\hn} \Psi \sim \left(\frac{N-2}{2} \right)^2\, \frac{\Psi}{r^2} + D(\beta) \Psi \quad  \text{ as } r \rightarrow 0^+.
$$
\end{lem}

\begin{proof}
The expression of hyperbolic Laplacian in radial coordinates, enables us to write

$$
- \Delta_{\hn} \Psi = - \Psi^{\prime \prime}(r) - (N-1) \coth r\, \Psi^{\prime}(r)  \quad \text{in } \hn \setminus\{x_0\}\,.
$$
Since, for $r>0$,
$$
\Psi^{\prime}(r) = \alpha \frac{\Psi(r)}{r} + \beta \coth r \Psi(r),
$$
and
$$
\Psi^{\prime \prime}(r) = \alpha(\alpha -1) \frac{\Psi(r)}{r^2} + 2 \alpha \beta \frac{\coth r}{r} \Psi(r)
+ \beta(\beta - 1) \frac{\Psi(r)}{\sinh^2 r} + \beta^2 \Psi(r)\, ,
$$
 we obtain

\begin{align*}
- \Delta_{\hn} \Psi & = - \left[ \alpha (\alpha -1) \frac{\Psi(r)}{r^2} +
 2 \alpha \beta \frac{\coth r}{r} \Psi(r) + \beta(\beta - 1) \frac{\Psi(r)}{\sinh^2 r} \right. \\
& \left. + (N-1) \alpha \frac{\coth r}{r} \Psi(r) + (N-1) \beta \Psi(r) + (N-1) \beta \frac{\Psi(r)}{\sinh^2 r}  + \beta^2 \Psi(r) \right]
\end{align*}
in $\hn \setminus\{x_0\}$.
Now, rearranging the above terms, the proof of \eqref{main-lemma-eq_1}
and \eqref{main-lemma-implication} follows directly by
substituting the value of $\alpha$ in \eqref{main-lemma-eq_1}.
\end{proof}

An application of Lemma \ref{main-lemma} yields
\begin{lem}\label{main-lemma2}
Let $N \geq 2$. For all $ \lambda \leq  \lambda_{1}(\hn)=\left(\frac{N-1}{2} \right)^2$ and $r>0$, set

$$
\Psi_{\lambda}(r) := r^{-\frac{N-2}{2}} \left(\frac{\sinh r}{r} \right)^{-\frac{N-1+\gamma_N(\lambda)}{2}}\,,
$$
where $\gamma_{N}(\lambda):=\sqrt{(N-1)^2-4\lambda}$. Then $\Psi_{\lambda}$ satisfies the following equation

\begin{align}\label{main-lemma-eq}
-\Delta_{\hn} \Psi_{\lambda} -\lambda \Psi_{\lambda}  =V_{\lambda}(r)  \Psi_{\lambda}\quad \text{in }\hn\setminus\{0\}\,,
\end{align}
with $V_{\lambda}(r)$ as given in \eqref{potential}.
\end{lem}

\begin{proof}
Let $g(r)$ be as defined in Theorem~\ref{improved-hardy}. Then \eqref{main-lemma-implication} can be rewritten as follows

\begin{align}\label{eq-main-lemma2}
 -\Delta_{\hn} \Psi & = \left( \beta + \frac{N-2}{2} \right)^2 \frac{\Psi}{r^2} + 2 \left( \beta + \frac{N-2}{2} \right) \left( \beta + \frac{N-1}{2} \right) g(r) \Psi  \notag \\
& - \beta(N-2 + \beta) \frac{\Psi}{\sinh^2 r} - \beta(N-1 + \beta)\Psi\,.
\end{align}
Now the proof follows by substituting $\beta = -\frac{N-1+\gamma_N(\lambda)}{2} $ in \eqref{eq-main-lemma2} and denoting by $\Psi_{\lambda}$ the corresponding function.
\end{proof}


\medskip

We now turn  to the criticality issue. We exploit \cite[Theorem~1.7]{pinch1} regarding a Liouville comparison principle for two nonnegative Schr\"odinger operators. For the reader convenience we quote below the theorem in the particular case where the principal part of the two operators is the Laplacian.

\begin{thm}\cite[Theorem~1.7]{pinch1}\label{pinch-thm}
Let $N \geq 1$ and $\Omega$ be a domain in $\mathbb{R}^{N}$ or any noncompact Riemannian manifold.  Consider two Schr\"odinger operators defined on $\Omega$
of the form
$$
Q_j := - \Delta  + W_j, \quad j = 0,1,
$$
such that $W_j \in L^{p}_{\loc} (\Omega; \mathbb{R})$ for some $p > \frac{N}{2}.$ \\
Assume that the following assumptions hold true:

\begin{enumerate}

\item The operator $Q_1$ is critical in $\Omega.$ Denote by $\Phi$ be its ground state.

\medskip
\item $Q_0$ is nonnegative in $\Omega$, and there exists a real function $\Psi \in H^{1}_{\loc}(\Omega)$ such that $\Psi_{+} \neq 0,$ and $Q_0 \Psi \leq 0$ in $\Omega,$
    where $u_+(x) := \max \{ 0, u(x) \}.$

\medskip

\item  The following inequality holds:

\begin{equation*}
(\Psi_+)^2(x) \leq C \Phi^2(x) \quad a.e. \ \mbox{in} \ \Omega,
\end{equation*}
where $C > 0$ is a positive constant. \\

Then the operator $Q_0$ is critical in $\Omega$ and $\Psi$ is its ground state.

\end{enumerate}
\end{thm}
Recently the above result is extended to more general settings. We refer to \cite{ABG} for further details.

\medskip

We have now all the tools necessary for the proof of our main theorem.


\medskip

\noindent{\bf Proof of Theorem~\ref{improved-hardy}.} The proof of the inequality  rests on supersolution technique. The construction of a supersolution (in fact a solution in the case at hand) for the desired equation
directly follows from Lemma~\ref{main-lemma2} which states that, for all $ \lambda \leq  \lambda_{1}(\hn)$, the function $ \Psi_{\lambda}$, as defined there, is a positive solution of \eqref{main-lemma-eq}. Moreover, $\Psi_{\lambda} \in H^{1}_{\loc} (\hn \setminus \{ x_0 \}),$ and hence the required inequality \eqref{improved-poinc-lambda} follows using
the Allegretto-Piepenbrink theorem \cite[Theorem~2.12]{Cycon}.

\medskip
Next, by invoking
 Theorem~\ref{pinch-thm}, we show that  $\Psi_{\lambda}$ is the ground state of $ -\Delta_{\hn} -\lambda -V_{\lambda}(r)$ with $V_{\lambda}(r)$ as given in \eqref{potential}.
For this, following the notation of Theorem~\ref{pinch-thm}, we consider the operators defined in $\hn \setminus \{ x_0 \}$:
$$
Q_j := -\Delta_{\hn}   + W_j, \quad j = 0,1,
$$
where
$$W_0 =  -\lambda - V_{\lambda} \quad  \mbox{and} \quad  W_1 =  -\lambda_{1}(\hn) - V_{\lambda_{1}(\hn)}\,.
$$
Clearly, $W_j \in L^{p}_{\loc} (\hn \setminus \{ x_0 \}; \mathbb{R})$ for any $p > 1.$
By \cite[Theorem~2.1]{BGG} we know that when $N\geq 3$ the operator $Q_1$ is critical in $\hn \setminus \{ x_0 \}$ and the corresponding ground state is
 $\Psi_{\lambda_1}(r)=\Phi (r) : =r^{-\frac{N-2}{2}}  \left(\frac{\sinh r}{r} \right)^{-\frac{N-1}{2}} $. The same statement holds for $N=2$ but we postpone the proof to the end of the section.  As concerns the operator $Q_0$,  by Lemma~\ref{main-lemma2},
 we know that $\Psi_{\lambda} (r)= r^{-\frac{N-2}{2}} \left(\frac{\sinh r}{r} \right)^{-\frac{N-1+\gamma_N(\lambda)}{2}}
 \in H^{1}_{\loc} (\hn \setminus \{ x_0 \})$ satisfies $Q_0 \Psi_{\lambda} = 0$.  Moreover, for all $N \geq 3$, we have
  \begin{equation}
\label{ e-17}
\left(\frac{\Psi_{\lambda}(r)}{\Phi(r)}\right)^2 =\left(\frac{r}{\sinh r}\right)^{\gamma_N(\lambda)}\leq 1\quad \text{for all } r>0\,,
\end{equation}
 since $\gamma_N(\lambda) \geq 0$ for $0 \leq \lambda \leq \lambda_1(\hn)$. Therefore, all the assumptions of Theorem~\ref{pinch-thm} are satisfied and we conclude that $Q_0$, namely the operator $-\Delta_{\hn} - \lambda - V_{\lambda}$, is critical in $\hn \setminus \{ x_0 \}$ for all $0 \leq \lambda \leq \lambda_1(\hn)$ and $\Psi_{\lambda}$ is its ground state.
 \medskip

To complete the proof we still have to show that the operator $Q_1$ is  also critical in dimension two. To this aim we show that the equation $Q_1 u=0$ admits a ground state in $\hnn\setminus \{x_{0}\}$, namely a positive solution of minimal growth in a neighborhood of infinity in $\hnn\setminus \{x_{0}\}$, see \cite[Section 1]{PT2}. When $N=2$ the function $\Phi$ defined above reads $\Phi (r) = \left(\frac{r}{\sinh r} \right)^{\frac{1}{2}} $. Let $f$ be a smooth radial function in $\hnn\setminus \{x_{0}\}$, also exploiting Lemma \ref{main-lemma}, one can verified that
$$-\Delta_{\hn} (\Phi(r) f(r))=\frac{1}{4}\left(1+\frac{1}{r^2}- \frac{1}{\sinh r^2}\right)\,\Phi(r) f(r)-   \left(f^{\prime \prime}(r)+ \frac{N-1}{r} f^{\prime}(r) \right) \,\Phi(r)\,.$$
From the above computations it follows that two linearly independent solutions of the equation $Q_1 u=0$ are given explicitly by :
$$
\Phi (r) = \left(\frac{r}{\sinh r} \right)^{\frac{1}{2}}  \quad \mbox{and} \quad  \overline \Phi (r) = \left(\frac{r}{\sinh r} \right)^{\frac{1}{2}} \log r,
$$
 hence
$\Phi $ is a positive global solution while $\overline \Phi$ changes sign. Since $\overline \Phi$ is a positive solution of $Q_1 u = 0$ near infinity of $\hnn\setminus\{x_{0}\}$ and
\[
\lim_{r \rightarrow 0} \frac{\Phi(r)}{|\overline \Phi(r)|}=\lim_{r \rightarrow +\infty} \frac{\Phi (r)}{\overline \Phi (r)} = 0\,,
\]
by \cite[Proposition 6.1]{pinch} we conclude that $\Phi$ is  a positive solution of minimal growth in a neighborhood of infinity in $\hnn\setminus\{x_{0}\}$ and hence a ground
state of the equation $Q_1 u=0$. Namely, $Q_1$ is critical in $\hnn\setminus\{x_{0}\}$. By \eqref{ e-17} and Theorem 4.3, it follows that $\Psi_\lambda$ is a ground state of $Q_0$. This completes the proof.\par

\medskip

\noindent{\bf Proof of Theorem~\ref{improved-hardy_different}.}
The proof is divided into two steps. It rests on the explicit construction of solutions and then using the result of \cite{pinch} we derive an optimal
Hardy weight for the related operator.  \\

{\bf Step 1 :} It is an immediate consequence of Lemma~\ref{main-lemma} that $u_0(r) : = (\sinh r)^{\frac{2-N}{2}}$ satisfies the following equation :

\begin{equation}\label{super_hardy_type}
Hu_0:= -\Delta_{\hn} u_0 \, - \, \frac{(N-2)^2}{4} \frac{1}{\sinh^2 r} u_0 \, - \, \frac{N(N-2)}{4} u_0 = 0.
\end{equation}
Now we shall construct a second solution. Let us define $v(r) = (\sinh r)^{\frac{N- 2}{2}} w(r),$ where $w$ solves \eqref{super_hardy_type}. Then, $v$ satisfies the following equation

$$
v^{\prime \prime}(r) + \coth r v^{\prime}(r) = 0.
$$
This immediately implies either  $v(r) = \log \left| \tanh \frac{r}{2}\right|$ or constant. Therefore two independent positive solutions of the equation $Hu=0$  are
$u_0(r) =  (\sinh r)^{\frac{2-N}{2}}$ and $ u_1(r) = - (\sinh r)^{\frac{2-N}{2}} \log \left| \tanh \frac{r}{2}\right|.$ \\

{\bf Step 2 :} Now we evoke  \cite{pinch} for the construction of an optimal Hardy weight involving two independent positive solutions.
Using the above two positive solutions of the equation $Hu=0$, we
obtain the following optimal Hardy weight (in the sense of \cite{pinch})
$$
W: = \frac{1}{4}\left| \nabla \log\left( \dfrac{u_1}{u_0} \right) \right|^2  = \frac{1}{4} \frac{1}{\sinh^2 r (\log (\tanh \frac{r}{2}))^2}\,,
$$
for the operator $H$. In particular, $H - W$ is critical.
\medskip

{\bf Sharpness:} To prove the sharpness of the constant $\frac{N(N-2)}{4},$ let us fix some notations. Denote the cone of all positive solutions
of the equation $P \, u = 0$ in $\hn$ by $\mathcal{C}_{P}(\hn),$ where $P$ denotes any second order elliptic operator.  Define for a nonnegative potential $V$,

$$
\lambda_{0}(P, V, \hn) : = \sup \{ \lambda \in \mathbb{R} : \mathcal{C}_{P - \lambda V}(\hn) \neq \emptyset \},
$$
and
$$
\lambda_{\infty}(P,  V,\hn) : = \sup \{ \lambda \in \mathbb{R} : \exists \, K \Subset \hn  \, {\rm s.t.} \,  \mathcal{C}_{P - \lambda V}(\hn \setminus K) \neq \emptyset \}.
$$
Clearly $ \lambda_{0}(P,  V,\hn) \leq \lambda_{\infty}(P,V,  \hn).$ \\
From \cite{pinch}, we also know that the above optimal Hardy weight $W$ satisfies

$$
\lambda_{0}(H-W,  W,\hn\setminus\{x_0\}) = \lambda_{\infty}(H -W, W, \hn\setminus\{x_0\}) = 0.
$$
Furthermore, $\lambda=1$ is the best constant in a neighborhood of infinity of $\hn$ for the inequality $H-\lambda W\geq 0$.

Since, $W(r) \rightarrow 0$ as $r \rightarrow \infty,$ we conclude that for any $\varepsilon>0 $
there exists a compact set $K_\varepsilon$ containing $x_0$ such that
$$0\leq \lambda_{0}(H,1,  \hn\setminus\{x_0\})
\leq  \lambda_{0}(H,1,  \hn\setminus K_\varepsilon) \leq \lambda_{0}(H- W,W,  \hn\setminus K_\varepsilon)+\varepsilon=\varepsilon.$$

Therefore, $\lambda_{0}(H,1,  \hn\setminus\{x_0\}) = 0.$ Hence, $\frac{N(N-2)}{4}$ is the best constant in the above sense, and the proposition is proved.


\section{Optimality issues: proof of Theorems \ref{maintheorem1} and  \ref{maintheorem2}.}\label{proof_thm_1}

\subsection{Proof of Theorem~\ref{maintheorem1} }

\medskip

Let $N=3$. It is a well known fact that the equality in the Hardy inequality \eqref{hardy} is never achieved  in $H^{1}(\hn )$ for any $N \geq 3$, hence the infimum for $I(0)$ is never achieved. Therefore, it is enough to consider the case $0 < \lambda \leq \lambda_1({\mathbb{H}^{3}}).$ Furthermore, we have already seen that $\hat{\lambda}^3=\lambda_1({\mathbb{H}^{3}})=1$, hence $I(\lambda)=\frac{1}{4}$ for every $0\leq \lambda \leq 1$, where $I(\lambda)$ is defined by \eqref{quotient}. For $0 < \lambda <1$ it is easy to see that minimizers do not exist. Indeed, suppose for some
$0 < \lambda_0 < 1 $ there exists a minimizer $u_{\lambda_0} \in H^{1}(\mathbb{H}^{3})$ for $I(\lambda_0)$, then  any $\bar \lambda$ with $\lambda_0 < \bar \lambda < 1$ yields

 \begin{align*}
 \frac{1}{4} & = \frac{\int_{\hn} |\nabla_{\hn} u_{\lambda_{0}}|^2 \ \dvh - \lambda_0 \int_{\hn} u^2_{\lambda_0} \ \dvh}{\int_{\hn} \frac{u^2_{\lambda_0}}{r^2} \ \dvh} \\
 & > \frac{\int_{\hn} |\nabla_{\hn} u_{\lambda_{0}}|^2 \ \dvh - \bar \lambda \int_{\hn} u^2_{\lambda_0} \ \dvh}{\int_{\hn} \frac{u^2_{\lambda_0}}{r^2} \ \dvh} \geq \frac{1}{4}\, ,
 \end{align*}
a contradiction. Alternatively, we note that the subcriticality of the operator $-\Delta_{\hn} - \lambda- I(\lambda) \,\frac{1}{r^2}$ for $\lambda_0 < \bar \lambda < 1$ readily implies nonexistence of minimizers. To complete the proof  it remains to show that minimizers do not exist also for $\lambda = 1$. From the proof of Theorem \ref{improved-hardy} we see that $\Psi_{\lambda_1({\mathbb{H}^{3}})}(r)= \frac{\sqrt{r}}{\sinh r}$ satisfies
\begin{align*}
  - \Delta_{\mathbb{H}^{3}} \Psi_{\lambda_1({\mathbb{H}^{3}})} - \Psi_{\lambda_1({\mathbb{H}^{3}})}  & = \frac{1}{4} \frac{\Psi_{\lambda_1({\mathbb{H}^{3}})}}{r^2} \quad \text{ in }\mathbb{H}^{3} \setminus \{ x_0 \},
\end{align*}
and $\Psi_{\lambda_1({\mathbb{H}^{3}})}$ is the unique ground state to the corresponding equation. Since  $\Psi_{\lambda_1({\mathbb{H}^{3}})} \notin  H^{1}(\mathbb{H}^{3} )$, $I(\lambda)$ does not admit a minimizer. This completes the proof.

\medskip

 \subsection{Proof of Theorem~\ref{maintheorem2}}
Let $N>3$.  We start by noting that from Corollary \ref{improved-hardy_1} we have $\hat{\lambda}^N \geq N-2$. Furthermore, $I(N-2)$ is not achieved. Indeed, the operator
$$
 - \Delta_{\hn} - \frac{(N-2)^2}{4} \frac{1}{r^2} + (N-2)
$$
is subcritical in $\mathbb{H}^{N} \setminus \{ x_0 \}$, since the proof of Theorem~\ref{improved-hardy} implies that the function
$\Psi_{N-2}(r)= r^{\frac{N-2}{2}} (\sinh r)^{2 - N}$ is
 a positive supersolution of the Euler-Lagrange equation associated to $I(N-2)$ and it is not the ground state.  \par
 As concerns the upper bound for $\hat{\lambda}^N$, it also follows as a corollary of Theorem~\ref{improved-hardy}. Indeed, since $\hat{\lambda}^N \geq N-2$, from the definition of $\hat{\lambda}^N$ we have

\[\begin{aligned}
\int_{\hn} |\nabla_{\hn}u|^2 \ \mathrm{d}v_{\hn}  &\geq \left( \frac{N-2}{2} \right)^2 \int_{\hn} \frac{u^2}{r^2} \ \mathrm{d}v_{\hn} + \hat{\lambda}^N \int_{\hn} u^2 \ \mathrm{d}v_{\hn} \\
& =  \left( \frac{N-2}{2} \right)^2 \int_{\hn} \frac{u^2}{r^2} \ \mathrm{d}v_{\hn} + (N-2) \int_{\hn} u^2 \ \mathrm{d}v_{\hn}\\  &+ (\hat{\lambda}^N - (N-2)) \int_{\hn} u^2 \ \mathrm{d}v_{\hn}.
\end{aligned}
\]
Therefore, combining the fact that $g(r) \leq \frac{1}{3}$ with the criticality issue of Theorem~\ref{improved-hardy}, we readily infer that
$$
\hat{\lambda}^N - (N-2) <\frac{(N-2)(N-3)}{6}.
$$

\medskip
\par

\noindent\textbf{Proof of (i).}

\medskip
\noindent For $0 < \lambda < \hat{\lambda}^N$, the operator  $ -\Delta_{\hn} - \frac{(N-2)^2}{4} \frac{1}{r^2} - \lambda$ is \emph{subcritical}
   in $\hn \setminus \{ x_0 \}$, and hence, there is no minimizer associated to the related functional inequality.

  \medskip

Assume now $\lambda = \hat{\lambda}^N.$ This is the most delicate case.
We adapt to our setting the Euclidean approach of constructing suitable ``subsolution" to show non-achievement of the Hardy constant, see for instance \cite{Adimurthi}.  Suppose, by contradiction, that there exists a minimizer $\hat u \in H^{1}(\hn)$ of $I( \hat{\lambda}^N) =  \frac{(N-2)^2}{4}$, then it satisfies the equation

 \begin{equation*}
 -\Delta_{\hn} \hat u - \hat \lambda^N \, \hat u- \frac{(N-2)^2}{4}\, \frac{\hat u}{r^2}=0 \quad \text{a.e. in } \hn\setminus\{0\}\,.
 \end{equation*}
 Without loss of generality, by standard symmetrization on the hyperbolic space, see \cite{ALB}, we may assume $\hat u$ is radial and nonnegative. Furthermore, being superharmonic, $\hat u$ turns out to be positive by the Strong Maximum Principle \cite{Gilb}.  \medskip

Let $\delta < -1/2$ and $\varphi_{\delta}(x)=|x|^{-\frac{N-2}{2}}\, \log(\frac{1}{|x|})^{\delta}\in H^1(B_{R})$ for $0<R<1/\mathrm{e}$. By computing in hyperbolic radial coordinates, for $r\in (0,R)$ one has
\[\begin{aligned}
& -\Delta_{\hn} \varphi_{\delta}- \hat \lambda^N \, \varphi_{\delta} - \frac{(N-2)^2}{4}\, \frac{\varphi_{\delta}}{r^2}\\ &=  \frac{\varphi_{\delta}}{r^2 (\log r)^2} \left[\delta (1-\delta)-\hat \lambda^N r^2 (\log r)^2 -\delta(N-1) r^2 \,\log r\right.\\&\left.+ (N-1) \left (\coth r -\frac{1}{r}\right) \left(\frac{N-2}{2} r (\log r)^2+\delta r\, \log r\right)  \right]\,.
\end{aligned}\]
Set now $0<R_1<1/\mathrm{e}$ (not depending on $\delta$!) such that for all $r < R_1$ we have
$$(N-1) \left (\coth \frac{1}{\mathrm{e}} -\mathrm{e}\right) \frac{N-2}{2}\, r\,(\log(r))^2 <\frac{1}{4}\,.$$
Since $ |\log r| \leq \frac{N-2}{2}  (\log r)^2$ for $r<1/\mathrm{e}$ and $N\geq 4$, for $r<R_1$ we infer
$$\begin{aligned}&\delta (1-\delta)-\hat \lambda^N r^2 (\log r)^2 -\delta(N-1) r^2 \,\log r\\ &+ (N-1) \left (\coth r -\frac{1}{r}\right) \left(\frac{N-2}{2} r (\log r)^2+\delta r\, \log r\right)\\ &\leq \delta (1-\delta)+ (N-1) \left (\coth \frac{1}{\mathrm{e}} -\mathrm{e}\right) \,\frac{N-2}{2} \, r (\log r)^2 (1-\delta)\\ &\leq \delta (\frac{3}{4}-\delta)+ \frac{1}{4}\leq -\frac{3}{8}\,.\end{aligned}$$
Hence,
$$-\Delta_{\hn} \varphi_{\delta}- \hat \lambda^N \, \varphi_{\delta} - \frac{(N-2)^2}{4}\, \frac{\varphi_{\delta}}{r^2} < 0\quad \text{in } B_{R_1}\setminus\{0\}\,. $$
Set $M(\delta)=\frac{\hat u(R_1)}{\varphi_{\delta}(R_1)}$ and $\psi_{\delta}(r):=\hat u(r)-M(\delta)\,\varphi_{\delta}(r)$. Then, $\psi_{\delta} \in H^{1}_{0}(B_{R_1})$ and satisfies
$$-\Delta_{\hn} \psi_{\delta}- \hat \lambda^N \, \psi_{\delta} - \frac{(N-2)^2}{4}\, \frac{\psi_{\delta}}{r^2} > 0\quad \text{in } B_{R_1}\setminus\{0\}\,.$$
 Set now $\psi_{\delta}^{-} = \min \{\psi_{\delta}, 0\}$ and $\psi_{\delta}^{+} = \max \{\psi_{\delta}, 0\}$, we have that $\psi_{\delta}^{-}, \psi_{\delta}^{+}\in H^{1}_{0}(B_{R_1})$. By multiplying the above inequality by $\psi_{\delta}^{-}$, using the fact $\psi_{\delta} = \psi_{\delta}^{+} + \psi_{\delta}^{-}$ and recalling the definition of $\hat{\lambda}^N$ in \eqref{hat}, we get
\begin{align*}
0\geq &\int_{B_{R_1}}\,\left(-\Delta_{\hn} \psi_{\delta}- \hat \lambda^N \, \psi_{\delta} - \frac{(N-2)^2}{4}\, \frac{\psi_{\delta}}{r^2} \right)\,\psi_{\delta}^{-} \,\dvh \\
=&\int_{B_{R_1}} |\nabla \psi_{\delta}^{-}|^2 \, \dvh - \hat\lambda^N \int_{B_{R_1}} (\psi_{\delta}^{-})^2 \, \dvh
-  \frac{(N-2)^2}{4} \int_{B_{R_1}} \frac{(\psi_{\delta}^{-})^2}{r^2} \, \dvh \geq 0 \,.
\end{align*}
Hence, $\psi_{\delta}^{-} = 0$. In particular, $\psi_{\delta} > 0$ and in turn $\hat u(r)>M(\delta)\,\varphi_{\delta}(r)$ for $0<r<R_1$. Finally, letting $\delta \rightarrow \bar \delta:=-\frac{1}{2}$ we readily get a contradiction since $M\left(\bar \delta\right)>0$ (due to the fact that $\hat u$ is positive and $R_1$ does not depend on $\delta$) and
$$+\infty > \int_{B_{R_1}} \frac{\hat u^2}{r^2}\, \dvh \geq  M^2\left(\bar \delta\right) \int_{B_{R_1}} \frac{\varphi_{\bar \delta}^2}{r^2}\, \dvh= +\infty\,.$$

\medskip

\noindent\textbf{Proof of (ii).} Let $\lambda < \lambda_1(\hn)$. Exploiting the
 Poincar\'e inequality \eqref{poincare}, in the sequel  we will endow the space $H^{1}(\hn)$ with the equivalent norm:

 \begin{equation}\label{equivalent}
 ||u||_{\lambda} :=  \left[ \int_{\hn} \left( |\nabla_{\hn} u|^2 - \lambda u^2 \right) \dvh  \right]^{\frac{1}{2}}, \quad u \in C_{c}^{\infty}(\hn)\,.
 \end{equation}
Since, by inequality \eqref{PH} we know that the
embedding $H^{1}(\hn) \hookrightarrow L^{2}(\hn, \frac{1}{r^2} \mbox{d}v_{\hn})$ is continuous but not compact, the existence of a minimizer to $I(\lambda)$ does not follow straightforwardly. When $\hat \lambda ^N<\lambda < \lambda_1(\hn)$, we overcome this difficulty by adapting to our setting the approach of \cite{EM}. To this aim, the crucial tool will be the following concentration compactness lemma in the hyperbolic setting.

\begin{lem}\label{CCL}
For $\lambda < \lambda_1(\hn)$, let $H^{1}(\hn)$ be endowed with the norm \eqref{equivalent}. Furthermore, let $\{ u_n \}$ be a bounded sequence in $H^{1}(\hn)$ such that: $u_n \rightharpoonup u \in H^{1}(\hn)$,
\begin{equation}\label{weaklimit1}
\left( |\nabla_{\hn} u_n|^2 - \lambda u_n^2 \right) \, \emph{d}v_{\hn} \rightharpoonup^* \mu  \quad \text{ and }\quad   \frac{u_n^2}{r^2}\, \emph{d}v_{\hn}  \rightharpoonup^* \nu\quad  \text{in the sense of measures.}
\end{equation}
Then, there holds

\begin{equation*}
\mu \geq \left( |\nabla_{\hn} u|^2 - \lambda u^2 \right) \ \emph{d}v_{\hn} + \mu_{0} \delta_{x_0}, \qquad 
\nu =  \frac{u^2}{r^2} \ \emph{d}v_{\hn} + \nu_{0} \delta_{x_0},
\end{equation*}
where $\delta_{x_0}$ is the Dirac measure centered at $x_0$, and $0 \leq \nu_0 \leq  \mu_0\,\frac{4}{(N-2)^2}$.
\end{lem}

\begin{proof}
The proof is divided into two steps.
\medskip

\noindent{\bf Step 1:} Since $\{ u_n \}$ is bounded in $H^{1}(\hn),$ \eqref{PH} implies that $\{ \frac{u_n}{r} \}$ is bounded in $L^2(\hn)$. Then,
$$ u_n \rightharpoonup u  \text{ in }H^{1}(\hn) \,,\quad u_n\rightarrow u  \text { in }L^{2}_{\loc}(\hn) \quad \text{ and }\quad  \frac{u_n}{r} \rightharpoonup  \frac{u}{r}  \text{ in }L^{2}(\hn)\,,$$
up to a subsequence. Denoting $v_n := u_n - u,$ it is then readily seen that
 $$ v_n \rightharpoonup 0  \text{ in }H^{1}(\hn) \,,\quad v_n\rightarrow 0  \text { in }L^{2}_{\loc}(\hn) \quad \text{ and }\quad  \frac{v_n}{r} \rightharpoonup 0  \text{ in }L^{2}(\hn)\,.$$
On the other hand, for $\Phi \in C_{c}^{\infty}(\hn)$ we have
\[\begin{aligned}
\int_{\hn} \Phi\, \left( |\nabla_{\hn} u_n|^2 - \lambda u_n^2 \right) \, {\rm d}v_{\hn}& =\int_{\hn}  \Phi\, \left( |\nabla_{\hn} u|^2 - \lambda u^2 \right) \, {\rm d}v_{\hn}\\ &+\int_{\hn}  \Phi\, \left( |\nabla_{\hn} v_n|^2 - \lambda v_n^2 \right) \, {\rm d}v_{\hn}\\
&+2 \int_{\hn}\Phi\,\left(\nabla v_n \cdot \nabla u +v_n u  \right)\, {\rm d}v_{\hn},
\end{aligned}
\]
 and
$$
\int_{\hn} \Phi \, \frac{ u_n^2}{r^2} \, \dvh = \int_{\hn} \Phi\, \frac{ u^2}{r^2}\, \dvh+ \int_{\hn} \Phi\, \frac{ v_n^2}{r^2} \, \dvh +2 \int_{\hn} \Phi\, \frac{ v_n\,u }{r^2}\, \dvh.
$$
Hence, taking the limit and recalling \eqref{weaklimit1}, for all $\Phi \in C_{c}^{\infty}(\hn)$ we conclude that
$$\int_{\hn} \Phi\, {\rm d}\mu=\int_{\hn}  \Phi\, \left( |\nabla_{\hn} u|^2 - \lambda u^2 \right) \, {\rm d}v_{\hn}+\int_{\hn} \Phi\, {\rm d}\mu_1,$$ 
and
$$
\int_{\hn} \Phi \, {\rm d}\nu = \int_{\hn} \Phi\, \frac{ u^2}{r^2}\, \dvh+\int_{\hn} \Phi \, {\rm d}\nu_1\,,
$$
where
$$
\left( |\nabla_{\hn} v_n|^2 - \lambda v_n^2 \right) \, {\rm d}v_{\hn} \rightharpoonup^* \mu_1  \quad \text{ and }\quad   \frac{v_n^2}{r^2}\, {\rm d}v_{\hn}  \rightharpoonup^* \nu_1\quad  \text{in the sense of measures.}$$

\medskip

\noindent{\bf Step 2:}
For what showed in Step 1, we may take $u=0$ in the following. Namely, we assume that $\{ u_n \}$ is such that
 $$ u_n \rightharpoonup 0  \text{ in }H^{1}(\hn) \,,\quad u_n \rightarrow 0  \text { in }L^{2}_{\loc}(\hn) \quad \text{ and }\quad  \frac{u_n}{r} \rightharpoonup 0  \text{ in }L^{2}(\hn), $$
and
$$
\left( |\nabla_{\hn} u_n|^2 - \lambda u_n^2 \right) \, {\rm d}v_{\hn} \rightharpoonup^* \mu ,  \quad \text{ and }\quad
\frac{u_n^2}{r^2}\, {\rm d}v_{\hn}  \rightharpoonup^* \nu\quad  \text{in the sense of measures.}
$$
For $\Phi \in C_{c}^{\infty}(\hn),$ we apply the Hardy inequality \eqref{hardy} to the functions $\{ \Phi u_n \}$ and we get

\begin{align*}
\frac{(N-2)^2}{4} \int_{\hn} \frac{(\Phi u_n)^2}{r^2} \, \dvh & \leq \int_{\hn} |\nabla_{\hn}(\Phi u_n)|^2 \, \dvh \\
& = \int_{\hn} \Phi^2 ( |\nabla_{\hn} u_n|^2 - \lambda u_n^2 ) \ \dvh + \int_{\hn} (\nabla_{\hn} \Phi)^2 u_n^2 \ \dvh \\
& +  2 \int_{\hn} (\Phi \nabla_{\hn} \Phi)(u_n \nabla_{\hn} u_n) \ \dvh + \lambda \int_{\hn} \Phi^2\, u_n^2 \ \dvh.
\end{align*}
Passing to the limit in the above inequality we conclude that
\begin{equation}\label{reverse_holder}
\int_{\hn} \Phi^2 \mbox{d}\nu \leq \,\frac{4}{(N-2)^2}\, \int_{\hn} \Phi^2 \mbox{d} \mu\, \quad \text{ for all } \Phi \in C_{c}^{\infty}(\hn).
\end{equation}
Then, a proper modification of the standard concentration compactness lemma \cite{Lions1, Lions2} yields that there exist a sequence of points ${x_i}\in\hn$ and two sequences of positive constants ${c_i}$ and ${\bar c_i}$ such that $\nu= \sum_{i=1}^{+\infty} c_i \delta_{x_i}$ and $\mu\geq \sum_{i=1}^{+\infty} \bar c_i \delta_{x_i}$, where $\delta_{x_i}$ is the Dirac measure centered at $x_i$. On the other hand, choosing $\Phi \in C_{c}^{\infty}(\hn)$ with supp$(\Phi) = K,$ a compact set, such that $x_0 \not \in K$,  we have that
$$
\int_{\hn} \frac{\Phi^2 u_n^2}{r^2} \ \dvh \leq C \int_{K} u_n^2 \, {\rm d}v_{\hn} \rightarrow 0 ,
$$
since $u_n \rightarrow 0$ in $L^2_{\loc}(\hn).$ Therefore the measure $\nu$ is only concentrated at $x_0$ and we conclude that $\nu=\nu_0 \delta_{x_{0}}$ and $\mu\geq \mu_0 \delta_{x_{0}},$ for some positive constants $\nu_0$ and $\mu_0$. Furthermore, inequality \eqref{reverse_holder}, together with standard measure theory arguments, gives
$$
\nu(B(x_0, \varepsilon)) \leq \,\frac{4}{(N-2)^2}\, \mu(B(x_0, \varepsilon)),
$$
where $\varepsilon > 0$ and $B(x_0, \varepsilon)$ denotes the ball with center $x_0$ and radius $\varepsilon.$ Finally, the arbitrariness of $\varepsilon$ implies
\[
\nu_0 \leq \frac{4}{(N-2)^2} \,\mu_0. \qedhere
\]
\end{proof}

The next lemma is devoted to study the  tightness of the measure defined above.

\begin{lem}\label{tightness}
In the same assumptions of Lemma \ref{CCL}, as $n \rightarrow \infty$ there holds
$$
\int_{\hn} \frac{u_n^2}{r^2} \, {\rm d}v_{\hn} = \int_{\hn} \frac{u^2}{r^2} \, {\rm d}v_{\hn} + \nu_0 + o(1).
$$
\end{lem}

\begin{proof}
By the definition of convergence of measure we have

\begin{equation}\label{convergence_measure}
\int_{\hn} \frac{u_n^2}{r^2} \, f \, \dvh = \int_{\hn} \frac{u^2}{r^2} \, f \, \dvh + f(x_0)  \nu_0 + o(1) \quad \forall f \in C_{c}^{\infty}(\hn).
\end{equation}

 Let $\varepsilon > 0, $ since $\frac{1}{r^2} \rightarrow 0$ as $r \rightarrow \infty,$ there exists $R_{\epsilon} > > 1$ such that
$\frac{1}{r^2} < \varepsilon$ $\forall \ r > R_{\varepsilon}.$ Thus

 \begin{align*}
\int_{\hn} \left( \frac{u_n^2}{r^2} - \frac{u^2}{r^2} \right) \, \dvh \,  -\nu_0 \,  & =  \int_{B(x_0, R_{\varepsilon})} \left( \frac{u_n^2}{r^2} - \frac{u^2}{r^2} \right) \, \dvh\,  -\nu_0 \\
&+
\int_{\hn \setminus B(x_0, R_{\varepsilon})}\left( \frac{u_n^2}{r^2} - \frac{u^2}{r^2} \right) \, \dvh  .
 \end{align*}
Using the fact that $\{ u_n \}$ is a bounded sequence in $H^{1}(\hn)$ together with Poincar\'e inequality, the last integral can be estimated as follows:

\begin{align*}
\int_{\hn \setminus B(x_0, R_{\varepsilon})}\left( \frac{u_n^2}{r^2} - \frac{u^2}{r^2} \right) \, \dvh  & \leq \varepsilon
 \left( \int_{\hn} u_n^2 \, \dvh + \int_{\hn} u^2 \, \dvh \right)  \leq C \varepsilon.
\end{align*}

Let us choose $\Psi \in C_{c}^{\infty}(\hn)$ such that $\Psi =1$ in $B(x_0, R_{\varepsilon})$ and \mbox{supp} $\Psi := B(x_0, 2R_{\varepsilon}),$ then we have

 \begin{align*}
 \left|  \int_{B(x_0, R_{\varepsilon})}\left(  \frac{u_n^2}{r^2} - \frac{u^2}{r^2}\right) \, \, \dvh  -\nu_0 \right|
 & \leq \left| \int_{B(x_0, 2 R_{\varepsilon})}  \Psi \left( \frac{u_n^2}{r^2} - \frac{u^2}{r^2} \right) \, \dvh - \nu_0 \right| \\
 & +  \left| \int_{B(x_0, 2R_{\varepsilon}) \setminus B(x_0, R_{\varepsilon})}  \Psi \left( \frac{u_n^2}{r^2} - \frac{u^2}{r^2} \right) \, \dvh \right|.
 \end{align*}
Therefore, using \eqref{convergence_measure} for the first term and the fact that $u_n\rightarrow u  \text { in }L^{2}_{\loc}(\hn)$ for the last term, we conclude that
$$
 \int_{B(x_0, R_{\varepsilon})} \left( \frac{u_n^2}{r^2} - \frac{u^2}{r^2} \right) \, \dvh \,  -\nu_0  \longrightarrow 0, \quad n \rightarrow \infty
$$
and the proof follows.
\end{proof}


We are now ready to prove the existence issue of Theorem~\ref{maintheorem2}-$(ii)$. For $\hat \lambda ^N<\lambda < \lambda_1(\hn)$, let $\{ u_n \}$ be a minimizing sequence for $I(\lambda)$ such that
\begin{equation*}
\int_{\hn} \frac{u_n^2}{r^2} \ \dvh = 1 \quad \text{ and } \quad \int_{\hn} ( |\nabla_{\hn} u_n|^2   - \lambda u_n^2) \  \dvh=I(\lambda)+o(1) \quad\text{ as } n\rightarrow +\infty\,.
\end{equation*}
Then, up to a subsequence, the assumptions of Lemma~\ref{CCL} are satisfied and using Lemma~\ref{tightness} one can write

\begin{align*}
&  I(\lambda) \left( \int_{\hn} \frac{u^2}{r^2} \ \dvh + \nu_0   \right) \geq  \int_{\hn} ( |\nabla_{\hn} u|^2   - \lambda u^2) \  \dvh +  \mu_0 +o(1) \quad\text{ as } n\rightarrow +\infty\,,
\end{align*}
and recalling the definition of $I(\lambda)$ we get $I(\lambda)\nu_0\geq \mu_0\geq \frac{(N-2)^2}{4} \nu_0$. Since $\lambda >\hat \lambda ^N$, we know that $I(\lambda) < \frac{(N-2)^2}{4}$. Therefore, we get a contradiction unless $\mu_0 = \nu_0 = 0.$ Hence, we get

$$
\int_{\hn} \frac{u^2}{r^2} \ \dvh = 1  \quad \text{ and } \quad \int_{\hn} ( |\nabla_{\hn} u|^2   - \lambda u^2) \  \dvh=I(\lambda)\,.
$$
Namely, $u \neq 0$ is a minimizer for $I(\lambda)$. As already remarked in the proof of $(i)$, up to replacing $u$ with $|u|$ and by maximum principle arguments, we may always assume that any minimizer has constant sign in $\hn\setminus \{x_0\}$. Once this noted, the uniqueness follows immediately. Otherwise, by taking a suitable linear combination of two minimizers, one may define a minimizer which changes sign, a contradiction.  \par

 To conclude the proof of statement $(ii)$, we still have to show the lower bound for $I(\lambda)$. Using Theorem~\ref{improved-hardy}, it follows that for any $\hat \lambda ^N<\lambda < \lambda_1(\hn)$ we have
  $$
 I(\lambda) \geq  \left(\frac{1+2\sqrt{\lambda_{1}(\hn)-\lambda}}{2}\right)^2.
  $$
 Since $I(\lambda)$ is achieved, the inequality is strict otherwise we contradict the criticality issue of Theorem~\ref{improved-hardy}.

\medskip

\noindent\textbf{Proof of (iii).} The proof relies on the fact that operator $L : = - \Delta_{\hn}  - \frac{(N-1)^2}{4} - \frac{1}{4r^2}$
is subcritical. Indeed, we have already remarked in the proof of Theorem \ref{improved-hardy} that $\Phi(r) = \left( \frac{r}{\sinh r} \right)^{\frac{N-1}{2}} r^{\frac{2-N}{2}}$ is a positive supersolution to the equation $L u =0$ in $\hn \setminus\{x_0\}$ which is not a solution.
Hence, a minimizer for $I(\lambda) = \frac{1}{4}$ cannot exist.


\section{General Cartan-Hadamard manifolds: proof of Theorems \ref{general manifold2}, \ref{improved-hardy_general}}\label{proof_general_manifolds}
We first recall some known facts.
Let $(M,g)$ be a Riemannian manifold. Take a point (pole) $x_0 \in M$ and denote $Cut \{x_0\}$ the cut locus of $x_0.$ We can
define the polar coordinates in $M \setminus Cut^*\{x_0\},$ where $Cut^*\{x_0\} = Cut\{x_0\} \cup \{ x_0\}.$ Indeed, to any point
$x \in M \setminus Cut^*\{x_0\}$ we can associate the polar  radius $r(x) := \mbox{dist}(x,x_0)$ and the polar angle
$\theta \in \mathbb{S}^{N-1},$ such that the minimal geodesics from $x_0$ to $x$ starts at $x_0$ to the direction $\theta.$ The Riemannian metric $g$ in $M \setminus Cut^*\{x_0\}$ in the polar coordinates takes the form
\[
 {\rm d}s^2 = {\rm d} r^2 + a_{i,j}(r,\theta) {\rm d}\theta_{i} \theta_{j},
\]
where $(\theta_{1},\ldots, \theta_{N-1})$ are coordinates on $\mathbb{S}^{N-1}$ and $((a_{i,j}))_{i,j= 1, \ldots, N}$ is a positive
definite Matrix. Let $a: = \det(a_{i,j}),$ $B(x_0,\rho) = \{ x := (r, \theta): r  < \rho \}. $ Then in $M \setminus Cut^*\{x_0\}$ we have
 \begin{equation}\label{radlaplacian1}
 \Delta_{M} = \frac{1}{\sqrt{a}} \frac{\partial}{\partial r}\left( \sqrt{a} \frac{\partial}{\partial r} \right)
+ \Delta_{\partial B(x_0,r)} = \frac{\partial^2}{\partial r^2} + m(r, \theta) \frac{\partial}{\partial r}
+ \Delta_{\partial B(x_0,r)},
 \end{equation}
where $\Delta_{\partial B(x_0,r)}$ is the Laplace-Beltrami operator on the geodesic sphere $\partial B(x_0,r)$ and $m(r, \theta)$
is a smooth function on $(0, \infty) \times \mathbb{S}^{N-1}$ which represents the mean curvature of $\partial B(x_0,r)$ in the
radial direction. For radial functions, namely functions depending only on $r$, if $M$ is the Riemannian model $M_{\psi}$ defined in Section \ref{manifolds}, the above expression reads

 \begin{equation}\label{radlaplacian}
 \Delta_{M_\psi} = \frac{1}{(\psi(r))^{N-1}} \frac{\partial}{\partial r} \left[ (\psi(r))^{N-1} \frac{\partial }{\partial r}
 (r) \right] = \frac{\partial^2}{\partial r^2}+ (N-1) \frac{\psi^{\prime}(r)}{\psi(r)} \frac{\partial}{\partial r}\,,
 \end{equation}
where $\psi$ is as defined by \eqref{curvmod}.

\medskip
The following Hessian comparison principle relates \eqref{radlaplacian1} and \eqref{radlaplacian}:

\begin{lem}\label{comp}\cite{RGR,AG}
Let $M$ be as in Theorem \ref{general manifold2}. The mean curvature of $\partial B(x_0,r)$ in the
radial direction satisfies

 \[
 m(r, \theta) \geq (N-1) \frac{\psi^{\prime}(r)}{\psi(r)}\quad \text{for all } r>0 \text{ and } \theta\in\mathbb{S}^{N-1}.
 \]
 \end{lem}

\medskip\noindent \bf Proof of Theorem \ref{general manifold2}. \rm We follow the same idea of the proof of Theorem~\ref{improved-hardy}, we define the function $\Psi_{\lambda}(r) := r^{-\frac{N-2}{2}} \left({\frac{\psi(r)}{r}}\right)^{-\frac{N-1+\gamma_N(\lambda)}{2}}$ and on the Riemannian model $M_{\psi}$ we compute
$$
-\Delta_{M_\psi} \Psi_{\lambda}(r)= \frac{(\gamma_{N}(\lambda)+1)^2}{4}\frac{\Psi_{\lambda}}{r^2} +V_{\psi}^{\lambda}(r)\, \Psi_{\lambda}\quad \text{in } M_{\psi}\setminus\{x_0\}\,,
$$
where $V_{\psi}^{\lambda}$ is defined as in the statement of Theorem \ref{general manifold2}. On the other hand, 
we claim that
$$\Psi_{\lambda}'(r)=\frac{\Psi_{\lambda}}{2r\psi}\left((1+\gamma_N(\lambda))\psi(r)-(N-1+\gamma_N(\lambda)) r\psi'(r) \right)\leq 0 \quad \text{for } r>0\,.$$
In fact, this is clearly true for $r$ close to zero, whereas we notice that
\[\begin{aligned}
&\left((1+\gamma_N(\lambda))\psi(r)-(N-1+\gamma_N(\lambda)) r\psi'(r) \right)'\\ &=-[(N-1+\gamma_{N}(\lambda))r\,\psi''(r)+(N-2)\psi'(r)] \leq 0
\end{aligned}\]
since $\psi$ is increasing and convex. Combining this fact with Lemma~\ref{comp}, by \eqref{radlaplacian1}  we infer
$$
 -\Delta_{M}  \Psi_{\lambda} \geq -\Delta_{M_\psi} \Psi_{\lambda}\,.
$$
Hence, since $\Psi_{\lambda} \in H^{1}_{\loc} (M \setminus \{ x_0 \})$,  the Allegretto-Piepenbrink theorem \cite[Theorem~2.12]{Cycon} implies  \eqref{generalHardylamb}.

\par

The proof of the criticality of the operator $ -\Delta_{M_{\psi}} - \frac{(\gamma_{N}(\lambda)+1)^2}{4r^2} - V_{\psi}^{\lambda}(r)$ in $M_{\psi} \setminus \{ x_0 \}$  follows by arguments similar to the one applied in the proof of Theorem \ref{improved-hardy}, hence we omit the details. We only mention that one has to exploit the fact that, by \cite[Theorem 2.5]{BGG}, it is known that the operator $ -\Delta_{M_{\psi}} - \frac{1}{4r^2} - V_{\psi}^{\lambda_1}(r)$ is critical in $M_{\psi} \setminus \{ x_0 \}$ and the corresponding ground state is $\Phi(r)=\Psi_{\lambda_1}(r)$, where $\lambda_1=(N-1)^2/4$. Finally, we note that since $\psi(r)> r$ for $r>0$ by construction, we can show that the quotient $\left(\frac{\Psi_{\lambda}(r)}{\Phi(r)}\right)^2$ is bounded, exactly as done in \eqref{ e-17} in the hyperbolic setting.\qed

\medskip\noindent \bf Proof of Theorem \ref{improved-hardy_general}. \rm We proceed with steps similar to the ones given in the proof of Theorem \ref{improved-hardy_different}. First, we notice that by a direct computation the function $u_0:=\psi^{\frac{2-N}2}$ is a solution to the equation
\begin{equation}\label{barrier}
-\Delta_{M_\psi}u-\frac{(N-2)^2}{4\psi^2}u-U_\psi\,u=0.
\end{equation}
We look for a second independent, positive solution to the same equation. If we set $v=\psi^{\frac{N-2}2}w$ with $w$ satisfying \eqref{barrier}, by another direct computation it turns out that $v$ must satisfy
\[
v''+\frac{\psi'}{\psi}v'=0.
\]
This yields that another positive, independent solution of equation \eqref{barrier} is found by taking, under the running assumptions,  \[v(r)=\frac1{\Theta(r)}\] with $\Theta$ as defined in \eqref{v}.
Notice that both solutions are decreasing. Hence they give rise to supersolutions of the corresponding equation on $M$ by Hessian comparison. We then perform the construction of the optimal weight given in \cite{pinch} starting from the two positive solutions $u_0$ and $u_0v$, see \cite{pinch}, thus yielding the stated inequality on $M$ and criticality on $M_\psi$.

\begin{rem}
It turns out from \cite[Proposition~3.1]{AG} that the subcriticality of $-\Delta_{M_\psi}$ in $M_{\psi}$ is equivalent to the following

\begin{equation}\label{sub_critical}
\int_{r}^{+\infty} \frac{1}{(\psi(s))^{N-1}} \, {\rm d}s < + \infty.
\end{equation}
Moreover, under \eqref{curv}, it can be shown easily that \eqref{sub_critical} is a weaker assumption than \eqref{v}. Therefore, assuming \eqref{sub_critical} we can
associate a natural Hardy weight for the operator $-\Delta_{M_\psi}.$ It is easy to see $u_0 = 1$ and $u_1 = \int_{r}^{+\infty} \frac{1}{(\psi(s))^{N-1}} \, {\rm d}s$ are
two independent positive solutions of

$$
-\Delta_{M_\psi} = 0 \quad \mbox{in} \quad M_{\psi}.
$$
Therefore using \cite{pinch}, an optimal Hardy weight $\tilde W_{\psi}$ for the operator $-\Delta_{M_\psi}$ is given by

$$
\tilde W_{\psi} = \frac{1}{4} \left( \dfrac{(\psi(r))^{1-N}}{\int_{r}^{+\infty} (\psi(s))^{1-N} {\rm d}s} \right)^2.
$$
In the hyperbolic space $\hn,$ $\tilde W_{\sinh}$ has the following asymptotic which should be compared with the statement of Theorem~\ref{improved-hardy} for
$\lambda=\lambda_{1}(\hn)$:
$$\tilde W_{\sinh}(r) \;\substack{\sim\\ r \rightarrow 0} \frac{(N-2)^2}{4 r^2} \ \mbox{ and } \ \tilde W_{\sinh}(r) \;\substack{\sim\\ r \rightarrow \infty} \, \frac{(N-1)^2}{4} + \frac{(N-1)(N+3)}{N+1} e^{-2r} + \circ(e^{-2r}).
$$
\end{rem}


\section{Weighted Hardy and Rellich inequalities: proof of Theorems~\ref{Rellich_lemma_1} and \ref{Rellich_lemma_2}}

\subsection{ Proof of Theorem \ref{Rellich_lemma_1}.}

We consider $u \in C_{c}^{\infty}(\hn \setminus \{ x_0 \})$ and define $ \dfrac{u(x)}{r^{\alpha}} = \Psi(x) v(x).$ Then we compute

$$
\frac{|\nabla_{\hn} u|^2}{r^{2\alpha}} = |\nabla_{\hn} \Psi|^2 v^2 + |\nabla_{\hn} v|^2 \, \Psi^2 + 2\, v \, \Psi  \, \langle \nabla_{\hn} \Psi,  \nabla_{\hn} v \rangle
+ 2\alpha \frac{u_{r} u}{r^{2 \alpha + 1}} - \alpha^2 \frac{u^2}{r^{2\alpha + 2}}.
$$
Now integrating above and by integration by parts we obtain

\begin{align*}
 \int_{\hn} \frac{|\nabla_{\hn} u|^2}{r^{2\alpha}}\, {\rm d}v_{\hn}  & = \int_{\hn} \frac{(- \Delta_{\hn} \Psi)}{\Psi} \frac{u^2}{r^{2 \alpha}} \, {\rm d}v_{\hn}
+  \int_{\hn} |\nabla_{\hn} v|^2 \, \Psi^2 \, {\rm d}v_{\hn}\\ &- \alpha^2 \int_{\hn} \frac{u^2}{r^{2 \alpha + 2}} \, {\rm d}v_{\hn} \\
& +  \alpha(2 \alpha + 1) \int_{\mathbb{S}^{N-1}} \left( \int_{0}^{\infty} \frac{u^2}{r^{2 \alpha + 2}} \, (\sinh r)^{N-1} \, {\rm d}r \right) {\rm d} \sigma  \\
& -(N-1) \alpha \int_{\mathbb{S}^{N-1}} \left( \int_{0}^{\infty} \frac{u^2}{r^{2 \alpha + 1}} \, \coth r \, (\sinh r)^{N-1} \, {\rm d}r \right) {\rm d} \sigma \\
& \geq  \int_{\hn} \frac{(- \Delta_{\hn} \Psi)}{\Psi} \frac{u^2}{r^{2\alpha}} \, {\rm d}v_{\hn} + (\alpha^2 -(N-2) \alpha) \int_{\hn} \frac{u^2}{r^{2\alpha + 2}} \, {\rm d}v_{\hn} \\
& -(N-1) \alpha \int_{\hn} \frac{g(r)}{r^{2 \alpha}}\, u^2 \, {\rm d}v_{\hn}.
\end{align*}
Here in particular we insert $\Psi := \Psi_{N-2}$, as defined in the proof of Theorem~\ref{improved-hardy}, in the above to obtain the desired result \eqref{eq_Hardy_lemma_1}. The sharpness of the constant follows
immediately by exploiting the test function of type  $\Phi(r) := r^{\frac{2-N}{2}}$ combining with a proper cut off as in the case of  classical Hardy inequality on the
hyperbolic space (see \cite[Theorem~3.1]{Yang}).

\medskip
 \subsection{ Proof of Theorem \ref{Rellich_lemma_2}.}

 The proof is based on a suitable combination of Theorem \ref{Rellich_lemma_1} with some ideas taken from the proof of \cite[Theorem~3.1]{Kombe2}. More precisely, starting from the inequality
$$-\Delta_{\hn} \frac{1}{r^{2\alpha}} \geq \frac{2\alpha(N-2 - 2 \alpha)}{r^{2 \alpha + 2}} \text{ for }  \alpha>0 \,,$$
multiplying both sides by $u^2\in C_{c}^{\infty}(\hn \setminus \{ x_0 \})$ and integrating over $\hn$, one gets
$$- \int_{\hn} \frac{u \Delta_{\hn} u}{r^{2\alpha}}\, {\rm d}v_{\hn} \geq   \int_{\hn} \frac{|\nabla_{\hn} u|^2}{r^{2\alpha}}\, {\rm d}v_{\hn} +
 \alpha( N- 2 - 2 \alpha)  \int_{\hn} \frac{ u^2}{r^{2 \alpha + 2}}\, {\rm d}v_{\hn}\,.$$
Then, for any $\varepsilon>0$, by Young's inequality there holds
$$\int_{\hn} \frac{ |\Delta_{\hn} u|^2}{r^{2 \alpha - 2}}\, {\rm d}v_{\hn} \geq  4\varepsilon  \int_{\hn} \frac{|\nabla_{\hn} u|^2}{r^{2\alpha}}\, {\rm d}v_{\hn}
 + (4\varepsilon \alpha (N-2 - 2\alpha)  - 4\varepsilon^2)\int_{\hn} \frac{ u^2}{r^{2 \alpha + 2}}\, {\rm d}v_{\hn}\,.$$
A combination of the above inequality with \eqref{eq_Hardy_lemma_1} yields
 \begin{align*}
 \int_{\hn}  \frac{|\Delta_{\hn}u|^2}{{r^{2 \alpha - 2}} }\, {\rm d}v_{\hn}& \geq  4\varepsilon \left( \frac{(N-2 - 2 \alpha)^2}{4}+ \alpha(N- 2 - 2 \alpha) -\varepsilon \right)
 \int_{\hn} \frac{ u^2}{r^{2 \alpha + 2}}\, {\rm d}v_{\hn}\\
 &+ 4 \varepsilon (N-2) \int_{\hn} \frac{u^2}{r^{2 \alpha}} \, {\rm d}v_{\hn}\\ &+ 4 \varepsilon  \left( \frac{(N-2)(N-3)}{2}  - (N-1) \alpha \right) \int_{\hn} \frac{g(r)}{r^{2\alpha}} u^2 \, {\rm d}v_{\hn}.
 \end{align*}
 Finally, by maximizing the coefficient in front of the first term on the right hand side, one gets $\varepsilon=\frac{(N- 2 - 2\alpha)(N-2 + 2 \alpha)}{8}$. By inserting this value in the above inequality, the proof follows.


\section*{Appendix: Improved hardy inequality in two dimensional Euclidean Space}
This section is devoted to state certain improved Hardy inequalities in two dimensional Euclidean space. The results can be obtained from a direct application of Theorem \ref{critical2} after suitable transformations.

Let $B$ be the Euclidean unit ball. From Theorem \ref{improved-hardy} and conformal invariance of the Dirichlet norm in dimension two, i.e.,
$\int_{\hnn} |\nabla_{\hnn} u|^2 \, {\rm d}v_{\hnn} = \int_{B} |\nabla u|^2\, {\rm d}x,$ where ${\rm d}x$ denotes the Euclidean volume element, we derive the following result.

\begin{cor}
For all $ \lambda \leq  \lambda_{1}(\hnn)=\frac{1}{4} $ and all $ u \in C^{\infty}_{0}(B)$ the following inequality holds
\begin{align}\label{ball}
&\int_{B} |\nabla u|^2 \, {\rm d} x - \lambda \int_{B} \left( \frac{2}{ 1 - |x|^2} \right)^2 u^2 \, {\rm d}x \geq \\  \notag
&  \frac{(\sqrt{1-4\lambda}+1)^2}{4}
\int_{B} \left(  \frac{1}{\left(\log\left(\frac{1-|x|}{1+|x|}\right)\right)^2} - \left(\frac{1-|x|^2}{2|x|}\right)^2   \right)\,  \left( \frac{2}{ 1 - |x|^2} \right)^2 u^2 \ {\rm d}x,  \\ \notag
&+\frac{(\sqrt{1-4\lambda})(\sqrt{1-4\lambda}+1)}{2}  \int_{B} \tilde g(|x|)\,  \left( \frac{2}{ 1 - |x|^2} \right)^2 u^2 \, {\rm d}x \,,
\end{align}
where ${\rm d}x$ denotes the Euclidean volume and $\tilde g(|x|):=g\left( \log\left(\frac{1-|x|}{1+|x|}\right)\right)$ with $g>0$ as defined in Theorem  \ref{improved-hardy}. In particular, for $\lambda=\lambda_{1}(\hnn)=\frac{1}{4}$, inequality \eqref{ball} reads as
$$\begin{aligned}
&\int_{B} |\nabla u|^2 \, {\rm d}x - \frac{1}{4} \int_{B} \left( \frac{2}{ 1 - |x|^2} \right)^2 u^2 \, {\rm d}x\\ &\geq  \frac{1}{4}
\int_{B} \left(  \frac{1}{\left(\log\left(\frac{1-|x|}{1+|x|}\right)\right)^2} - \left(\frac{1-|x|^2}{2|x|}\right)^2   \right)\,  \left( \frac{2}{ 1 - |x|^2} \right)^2 u^2 \, {\rm d}x,
\end{aligned}$$
and the constant $1/4$ in the right hand side in the above inequality is sharp.
\end{cor}

\begin{rem}
The inequality \eqref{ball} can be compare with \emph{optimal Leray inequality},  (cf. \cite[Example~13.2]{pinch}). Note that the weight in the left hand side of \eqref{ball} has a
singularity only at the boundary of the ball but on the other hand we have a subcriticality of the resulting operator unlike in the case of classical
\emph{Leray inequality}.
\end{rem}

By considering the upper half space model for $\hnn$, namely $\mathbb{R}^{2}_{+} = \{ (x, y) \in \mathbb{R} \times \mathbb{R}^{+} \} $ endowed with the Riemannian metric $\frac{\delta_{ij}}{y^2}$.
Theorem \ref{improved-hardy}  yields the following improved Hardy-Maz'ya-type inequality in the half space in dimension two:

\begin{cor}
For all $ \lambda \leq  \lambda_{1}(\hnn)=\frac{1}{4} $ and all $u \in C_{c}^{\infty} (\hnn \setminus \{ x_0 \})$ there holds
\begin{equation}
\begin{aligned}\label{mazya1A}
&\int_{\mathbb R^{+}} \int_{\mathbb{R}} |\nabla u|^2 \, {\rm d}x \, {\rm d}y - \lambda \int_{\mathbb{R}^{+}} \int_{\mathbb{R}} \frac{u^2}{y^2} \, {\rm d}x \, {\rm d}y\\  & \geq
  \frac{(\sqrt{1-4\lambda}+1)^2}{4} \int_{\mathbb{R}^{+}} \int_{\mathbb{R}}  \left( \frac{1}{ d^2} - \frac{1}{\sinh^2 d} \right) \frac{u^2}{y^2} \, {\rm d}x \, {\rm d}y \\ \notag
&+\frac{(\sqrt{1-4\lambda})(\sqrt{1-4\lambda}+1)}{2} \int_{\mathbb{R}^{+}} \int_{\mathbb{R}} g(d)\, \frac{u^2}{y^2} \, {\rm d}x \, {\rm d}y ,
\end{aligned}
\end{equation}
where $(x, y) \in \mathbb{R} \times \mathbb{R}^{+}$, $d=d(x,y): = \cosh^{-1} \left( 1 + \frac{(y - 1)^2 + |x|^2}{2 y} \right)$ and $g>0$ is as defined in Theorem  \ref{improved-hardy}. \par
In particular, for $\lambda=\lambda_{1}(\hnn)=\frac{1}{4}$, inequality \eqref{mazya1A} reads as
\begin{equation}\label{mazya1}
\int_{\mathbb R^{+}} \int_{\mathbb{R}} |\nabla u|^2 \, {\rm d}x \, {\rm d}y - \frac{1}{4} \int_{\mathbb{R}^{+}} \int_{\mathbb{R}} \frac{u^2}{y^2} \, {\rm d}x \, {\rm d}y \geq \frac{1}{4} \int_{\mathbb{R}^{+}} \int_{\mathbb{R}}  \left( \frac{1}{ d^2} - \frac{1}{\sinh^2 d} \right) \frac{u^2}{y^2} \, {\rm d}x \, {\rm d}y,
\end{equation}
and the constant $1/4$ in the right hand side of \eqref{mazya1} is sharp.
\end{cor}
Hence, \eqref{mazya1} provides an optimal nonstandard remainder term for the Hardy-Maz'ya inequality in dimension two, see \cite{BGG} for the case $N\geq 3$.

\medskip
\noindent {\sc Acknowledgment.}  E.B. has been partially supported by the Research Project FIR (Futuro in Ricerca) 2013 \emph{Geometrical and qualitative aspects of PDE's}. D.G. has been partially supported at the Technion by a fellowship of the Israel Council for Higher Education. G.G.~has been partially supported by the PRIN project {\em Equazioni alle derivate parziali di tipo ellittico e parabolico: aspetti geometrici, disuguaglianze collegate, e applicazioni} (Italy). Both E.B. and G.G. ~have also been supported by the Gruppo Nazionale per l'Analisi Matematica, la Probabilit\`a e le loro Applicazioni (GNAMPA) of the Istituto Nazionale di Alta Matematica (INdAM, Italy).  Y.P. and D.G. acknowledge the support of the Israel Science Foundation (grants No. 970/15) founded by the Israel Academy of Sciences and Humanities.



\end{document}